\documentclass[11pt,a4paper]{article} 
\usepackage{pgfplots}
\pgfplotsset{compat=1.18}
 \usepackage{hyperref}

\usepackage{a4wide,amsmath,amssymb,amsthm,mathtools,
xcolor}
\usepackage[all]{xy}
\usepackage{tikz}
\usetikzlibrary{cd}
\usepackage[mathcal]{euscript}
\usepackage{mathrsfs}
\usepackage{enumitem}
\usepackage[capitalise]{cleveref} 
\let\mathcal=\mathscr

\newtheorem{thm}{Theorem}[section]
\newtheorem{corol}[thm]{Corollary}

\newtheorem{lemma}[thm]{Lemma}

\newtheorem{prop}[thm]{Proposition}
\newtheorem{defin}[thm]{Definition}

\newtheorem{notation}[thm]{Notation}
\theoremstyle{remark}
\newtheorem{rem}[thm]{Remark}
\newtheorem{ex}[thm]{Example}
\newenvironment{remark}{\begin{rem}\rm}{\parbox{2mm}{\hfill}\qee\end{rem}}
\newenvironment{example}{\begin{ex}\rm}{\parbox{2mm}{\hfill}\hfill\qee\end{ex}}

\newcommand{\cO}{{\mathcal O}}

\newcommand{\PP}{{\mathbb P}}
\newcommand{\BA}{{\mathbb A}}

\newcommand{\E}{\mathfrak E}
\newcommand{\F}{\mathfrak F}
\newcommand{\gS}{\mathfrak S}
\newcommand{\gQ}{\mathfrak Q}

\newcommand{\grass}[2]{\operatorname{Gr}_{#1}(#2)}

\newcommand{\SL}{\operatorname{SL}}
\newcommand{\Hom}{\operatorname{Hom}}
\newcommand{\End}{\operatorname{End}}
\newcommand{\Spec}{\operatorname{Spec}}
\newcommand{\Mat}{\operatorname{Mat}}
\newcommand{\Proj}{\operatorname{Proj}}

\newcommand{\red}{\operatorname{red}}

\newcommand{\rk}{\operatorname{rk}}

\newcommand{\C}{{\mathbb C}}

\newcommand{\Z}{{\mathbb Z}}

\renewcommand{\P}{{\mathbb P}}

\newcommand{\fE}{{\mathfrak E}}
\newcommand{\fFl}{\operatorname{\mathfrak Fl}}
\newcommand{\Fl}{\operatorname{Fl}}

\newcommand{\qee}{\mbox{\hspace{0.2mm}}\hfill$\triangle$}
\newcommand{\id}{\operatorname{id}}

\newcommand{\Span}{\operatorname{Span}}
\newcommand{\Quot}{\operatorname{Quot}}
\newcommand{\tr}{\operatorname{tr}}

\newcommand\Hgrass{{\mathfrak{Gr}}} 
\newcommand{\hgrass}[2]{\operatorname{\mathfrak{Gr}}_{#1}(#2)}

\begin{document}
\begin{center} \bf   HIGGS GRASSMANNIANS \end{center}

\thispagestyle{empty} \vspace{-3mm}
\begin{center}{\sc Ugo Bruzzo,$^{ad}$ Michele Graffeo,$^{b}$  and Beatriz Gra\~na Otero$^{c}$}  \\[5pt]
\small
$^a$ Departamento de Matem\'atica, Instituto de Ci\^encias Exatas, 
Universidade Federal  \\ de Minas Gerais, Av.~Ant\^onio Carlos 6627, 
Belo Horizonte, 31270-901 MG, Brazil \\
$^b$ Dipartimento di Matematica, Politecnico di Milano, \\Piazza Leonardo da Vinci 32, 20133 Milano, Italy \\
$^c$ Departamento de Matem\'aticas  and IUFFYM (Instituto de F\'\i sica \\ Fundamental y Matem\'aticas), Universidad de Salamanca, \\
Plaza de la Merced 1-4, 37008 Salamanca, Spain \\
$^d$ IGAP (Institute for Geometry and Physics), Trieste, Italy  
\\[3pt]
Email: {\tt ubruzzo@ufmg.br, michele.graffeo@polimi.it, beagra@usal.es}

\end{center}

\vfill

\begin{abstract}   
We review and study the notion of Higgs Grassmannians, which are schemes parametrizing the Higgs subbundles of a given Higgs bundle over a smooth variety.   We write their   equations as  closed subschemes of the usual Grassmann bundles and investigate their geometry. Often the Higgs Grassmannians generically  have 0-dimensional fibers over the base variety, thus implying that Higgs subbundles are ``scarce.'' 
We characterize the structure of the Higgs Grassmannians  by analyzing the local Jordan type of the Higgs field. A refined analysis of the rank 2 case is also provided in terms of the discriminant of the characteristic polynomial. We apply our characterizations to the Simpson system of a smooth variety to provide a streamlined proof of its semistability, and we establish a structural relationship between the rank 1 Higgs Grassmannian and the spectral cover of the   Higgs bundle. Finally, we introduce the schemes of flags of Higgs subbundles of a given Higgs bundle, and the Quot schemes parametrizing Higgs quotients; we conclude with some examples.
\end{abstract}

\vfill\begin{minipage}{\textwidth} \small
\parbox{\textwidth}{\hrulefill} \\
Date: \today \\
MSC 2020: 14H60, 14J60, 14M15 \\
Keywords: Higgs bundles, Grassmann bundles, semistability, spectral covers \\
Research partly supported by Bolsa de Produtividade 305343/2025-4 of Brazilian CNPq (UB), PRIN 2022BTA242 ``Geometry of algebraic structure: moduli, invariants, deformations,''  (UB-MG), INdAM-GNSAGA (MG), Grant PID2021-128665NB-I00 funded by MCIN/AEI/10.13039/501100011033 (Spain),   ``ERDF A way of making Europe,''  Universidad de Salamanca through Programa XIII (BGO).

\end{minipage}

\newpage

\section{Introduction}
A Higgs bundle, or more generally, a Higgs sheaf, is an example of a {\em decorated sheaf,} i.e., a coherent sheaf equipped with an additional structure; given a smooth variety $X$, a Higgs sheaf on $X$ is a pair $\mathfrak F=(F,\phi)$, where $F$ is a coherent sheaf on $X$, and $\phi\colon F \to F\otimes\Omega^1_X$ is an $\cO_X$-linear morphism. When $\dim X>1$, $\phi$ is subject to the following condition: after locally trivializing the cotangent sheaf $\Omega^1_X$, the morphism $\phi$   corresponds to $n$ endomorphisms $(\phi^{(1)},\dots,\phi^{(n)})$ of $F$, where $n = \dim X$; one requires these endomorphisms to commute with each other. This makes the spectral theory of $\phi$ manageable.

For any decorated sheaf, it is quite natural that any additional structure or property of the sheaf is required to be compatible with the decoration. So, for instance, in defining the notion of slope (semi)stable Higgs sheaf, one  asks  the inequality between the slopes  to hold only for {\em $\phi$-invariant subsheaves,} i.e, subsheaves $G$ of $F$ such that $\phi(G)\subset G\otimes\Omega_X^1$ ((semi)stability of Higgs sheaves was 
introduced by Hitchin for the locally-free rank-two case on Riemann surfaces \cite{Hitchin-self} and generalized to coherent Higgs sheaves of arbitrary rank and base variety dimension 
by Simpson \cite{simpson-local}).  Note that a Higgs sheaf $\mathfrak F=(F,\phi)$ may be (semi)stable without its underlying sheaf $F$ being (semi)stable.

Restricting ourselves to the locally free case, and in view of the 
relevance of invariant subbundles of a given Higgs bundle, the question 
naturally arises as to whether there exists a scheme that parametrizes 
the Higgs subbundles of a Higgs bundle --- or more precisely, represents the functor
 of families of Higgs subbundles --- just as Grassmann bundles represent 
the functor of families of subbundles of a given vector bundle. It comes to no surprise that such a scheme exists, and in a natural way   is a closed subscheme of the relevant Grassmann bundle.

To be more precise, if $\fE=(E,\phi)$ is a Higgs bundle on $X$ --- that is, a Higgs sheaf whose underlying coherent sheaf $E$ is locally free --- we shall denote by $\grass dE$ the Grassmann bundle of $d$-planes in $E$, where $1\le d < r=\rk E$. Substantially following \cite{bruzzo-hernandez-dga}, we shall construct a scheme $\hgrass d{\fE}$, that we call the $d$-th Higgs Grassmannian of $\fE$. This scheme represents the functor of families of rank $d$ Higgs subbundles of $\fE$; this will be a closed subscheme of $\grass dE$. The purpose of this paper is to study the schemes $\hgrass d{\fE}$, paying particular attention to the case of rank 1 subbundles ($d=1$). We shall see that in many instances the dimension of the generic fiber of  $\hgrass 1{\fE}$ over $X$ is   0,   showing that Higgs subbundles are scarce; some Higgs bundles, as we shall see, may not have Higgs subbundles at all.

\paragraph{Organization of the content.} We highlight the contents of the various Sections. The notion and   formal definition of the Higgs Grassmannians are recalled in \Cref{sec:equations}; the equations of the  Higgs Grassmannians of a Higgs bundle $\fE=(E,\phi)$ as closed subschemes of the Grassmann bundles of $E$ are written, first in the case of rank 1 subbundles, and then in the general case. 

In \Cref{sec:curves} we study the structure of the Higgs Grassmannian of rank 1 subbundles when the base variety is a smooth curve, for Higgs bundles of any rank. In this analysis we take advantage of the fact that the Jordan type of the Higgs field is constant in an open subset. The Section ends with a few simple examples. 

In \Cref{sec:dimnrkrsingleJb} we allow the base variety to be of arbitrary dimension. Now locally the Higgs field consists of a number of commuting matrices, so that the Section starts with a result in linear algebra: if one of the matrices has a single Jordan block, the remaining ones can be set in a triangular Toeplitz  form. Building on this, we first characterize the Higgs Grassmannian of rank 1 subbundles when the Higgs fields has that structure, then extending this to the case of several Jordan blocks with distinct eigenvalues.

In \Cref{section:discriminant} we make a somewhat more refined analysis, but considering only the rank 2 case, in terms of the discriminant of characteristic polynomial of one of the matrices that locally  make up the Higgs field; a complete classification is given for the case of a one-dimensional base.

There is a Higgs bundle naturally associated with any smooth variety $X$, which we call the {\em Simpson system} of $X$, i.e., $\gS=(S, \phi)$ with $S=\cO_X \oplus \Omega^1_X$ and $\phi(f, \omega)= (\omega, 0) $; this was introduced by Simpson in \cite{simpson-unif} and was showed to be  stable when $X$ is a surface of general type, thus providing a  proof of the Miyaoka-Yau inequality in that case. This was generalized to higher dimension in \cite{GKPT19}. In 
\Cref{section:Simpson} we characterize the Higgs Grassmannians of the Simpson system, and use that characterization to provide a quick proof of the semistability of the Simpson system,  relying on the semistability of the tangent bundle of $X$. 

In \Cref{section:spectral} we review the definition of spectral cover $S_{\fE}$ of a Higgs bundle $\fE$, and show that there is a surjective morphism $g \colon \hgrass{1}{\fE} \to S_{\fE}$, which admits a section over the smooth locus of $S_{\fE} $. 

In view of further developments, in \Cref{sec:flag} we introduce the schemes representing the functor of flags of Higgs subbundles and the Higgs Quot functor, discussing  three motivating examples.

\paragraph{Acknowledgement.} We thank Daniel Hern\'andez Ruip\'erez for useful discussions. BGo thanks SISSA, and MG thanks Universidade Federal de Minas Gerais in Belo Horizonte for hospitality and support.

\section{Higgs Grassmannians}\label{sec:equations}

We recall from \cite{bruzzo-hernandez-dga} the definition of the Higgs Grassmannians (see also \cite{bruzzo-grana-adv}); note however that there the authors characterized the Grassmannians of Higgs quotients, while in this paper we  prefer to deal with Higgs subbundles. Moreover, in this Section we shall write down the equations of the Higgs Grassmannians as closed subschemes of the Grassmann bundles $\grass{d}{E}$.

\subsection{Definition of the Higgs Grassmannian}

Let $X$ be an $n$-dimensional smooth variety over $\C$. We recall that a Higgs sheaf on $X$
is a pair $\F=(F,\phi)$, where $F$ is a coherent $\cO_X$-module, and $\phi\colon F \to F\otimes\Omega^1_X$ is an 
$\cO_X$-linear morphism such that the composition $$\phi\wedge\phi \colon F \xrightarrow{\phi} F\otimes\Omega^1_X  \xrightarrow{\phi\otimes\mathrm{id}}   F\otimes \Omega^1_X\otimes\Omega^1_X \xrightarrow{\mathrm{id}\otimes\wedge} F\otimes\Omega^2_X$$
is zero (this is the intrinsic form of the commutativity mentioned in the Introduction, and is trivially verified when $\dim X =1$).  A Higgs bundle is a Higgs sheaf with $F$ locally free. Semistability and stability of   torsion-free Higgs sheaves are defined as for $\cO_X$-modules, but only with reference to $\phi$-invariant subsheaves, i.e., $\mu(G) \le \mu(F)$, or $\mu(G) < \mu(F)$, whenever $G$ is a $\phi$-invariant subsheaf of $F$, with $0 < \rk G < \rk F$; the slope $\mu$ is defined as $\mu(F)= \deg F/\rk F$. In turn, $\deg F=c_1(F) \cdot H^{n-1}$, where $H$ is a fixed polarization.

 For a given   rank $r$ vector bundle $E$ on  $X$, and for every $d$ in the range $ 0 < d < r$, 
 we denote the Grassmann bundle on $X$ of rank $d$ subbundles  of $E$ as  $\grass dE$.  
 Since $\grass1E=\P E$ we shall use the latter notation. One has the universal exact sequence  
\begin{equation} 0 \to S_d  \to \pi_d^\ast E \to Q_d \to 0 \label{eq:univ}\end{equation} of vector bundles
on $\grass dE$, where $S_d$ is the universal rank $d$ subbundle, 
$Q_d$ is the rank $r-d$ universal quotient bundle, and   $\pi_d\colon\grass dE \to X$ is the projection.
If $\E=(E,\phi)$ is a Higgs bundle, we form the diagram
\begin{equation*}\label{diagram}\xymatrix{0\ar[r]&S_d  \ar[r]^{{a}_d}& \pi^\ast_d   E \ar[r]\ar[d]^{\pi^\ast_d \phi}& Q _d  \ar[r]&0\\
0\ar[r]&S_d  \otimes \Omega^1_{\grass dE}\ar[r]&\pi^\ast_d   E\otimes \Omega^1_{\grass dE}\ar[r]^{{b}_d}& Q _d  \otimes \Omega^1_{\grass dE}\ar[r]&0.}
\end{equation*} 
The $d$-th Higgs Grassmannian of $\E$, denoted $\hgrass d{\E}$, is the subscheme of $\grass d E$ defined by the zero locus of the composition 
${b}_d\circ\pi_d^\ast\phi\circ {a}_d$. By construction, the restrictions of the bundles $S_d$ and $Q_d$ to   $\hgrass d{\E}$ carry Higgs fields
induced by $\pi_d^\ast\phi$, so that we have an exact sequence of Higgs bundles on $\hgrass d{\E}$ 
\begin{equation}\label{eq:uniSeq}
0 \to \gS_d \to \rho_d^\ast \E \to \gQ_d \to 0    
\end{equation}
where $\rho_d$ is the restriction ${\pi_d}_{|\hgrass d{\E}}$. 

The scheme  $\hgrass d{\E}$ may be singular, reducible, nonreduced, nonequidimensional. On the positive side it enjoys
the analogous universal property of the usual Grassmann bundles:
if $f\colon Y \to X$ is a scheme morphism, and $\mathfrak{F}$ is a rank $d$    Higgs subbundle of
$f^\ast\E$, there is a morphism $g\colon Y \to \hgrass d{\E}$ such that
$\mathfrak{F}=g^\ast \gS_d$, and the diagram
$$\xymatrix{
& \hgrass d{\E} \ar[d]^{\rho_d} \\
Y\ar[ur]^g\ar[r]_f & X}
$$
commutes. 
\begin{remark}
\label{rem:universe}
As foreshadowed in the Introduction, the Higgs Grassmannian $\hgrass d{\E} $ is the scheme that represents the functor
    \begin{equation}\label{eq:functor}
            \begin{tikzcd}[row sep =tiny]
        \text{Sch}^{\mathrm{op}}  \arrow[r,"F_{d,\fE}"]&   \text{Sets}\\
        T\arrow[r,mapsto] & \{ \text{rank $d$ Higgs subbundles of $p_2^\ast \fE$}\},
    \end{tikzcd}
    \end{equation}
where $p_2\colon T\times X \to X$ is the projection onto the second factor.  We shall use this interpretation in \Cref{sec:flag} to simplify the definition of Higgs flag scheme.
\end{remark}

\subsection{Equations of the scheme of rank-one Higgs subbundles\label{eqns1}} 
We are going to write down the equations of the Higgs Grassmannians of a Higgs bundle inside the Grassmann bundle of the underlying vector bundle.
We start  with the case of $\hgrass 1 {\fE}$, the Higgs Grassmannian
of rank one Higgs subbundles of a 
rank $r$ Higgs vector bundle  $\fE $  on a $n$-dimensional smooth  
variety $X$, which was already treated in \cite{bruzzo-hernandez-dga}.  
We have a closed immersion $j \colon 
\hgrass 1 {\fE}\hookrightarrow \PP E$, and the universal Higgs subbundle 
is $\gS_1=\cO_{\hgrass 1 {\fE}}(-1)=j^\ast  \cO_{\PP E}(-1)$. We denote by 
$\rho_1\colon \hgrass 1 {\fE}\to X$ the projection. 
 
One can write local equations for $\hgrass 1 {\fE}$ by using the  Euler   sequence 
$$ 
\begin{tikzcd}
    0 \arrow[r]& \cO_{\PP E}(-1) \arrow[r,"\psi"]&   
\pi_1^\ast E \arrow[r,"\eta"]& T_{\PP (E)/X}\otimes\cO_{\PP E}(-1) \arrow[r]& 0\,, 
\end{tikzcd}
$$ 
as this is the form that \eqref{eq:univ} takes in this case. The Higgs Grassmannian 
$\hgrass 1 {\fE}$    is defined as the closed subscheme of $\PP E$ where the  composition of morphisms  
$$ 
\begin{tikzcd}[column sep = 4cm]
  \cO_{\PP E}(-1) \arrow[r,"    (\eta\otimes 1)\circ\pi_1^\ast\phi\circ\psi "]& 
 T_{\PP (E)/X}\otimes\cO_{\PP E}(-1)
\otimes \pi_1^\ast\Omega^1_X 
\end{tikzcd}
$$
vanishes. Given a local frame $\{e_1,\dots,e_r\} \subset H^0(U,E)  $ of $E$  on some trivializing open $U\subset X$, which induces  local 
vertical homogeneous coordinates $(z_1, \cdots, z_r)$ for $\PP E$,  
the Higgs field is  
represented by a matrix 
$(\phi_{ij})_{i,j=1}^r$ of 1-forms by letting 
$\phi(e_j)=\sum_{i=1}^r\phi_{ij}\otimes e_i$.
By computing explicitly the morphism $(\eta\otimes 1)\circ\pi_1^\ast\phi\circ\psi$ one obtains 
the  
homogeneous equations for 
$\hgrass 1 {\fE}$ in the form 
\begin{equation} 
\label{e:piphi}
\sum_{k=1}^r  
z_k(\phi_{jk}z_i-\phi_{ik}z_j)=0\,,\mbox{  for every }1\le i<j\le r. 
\end{equation}
So $\hgrass 1{\fE}$ is locally the intersection of $n\binom r2$ hypersurfaces in $\PP E$
that restricted to the fibers of $\PP E$ yield quadric hypersurfaces.   This number is  consistent with the more general formula we provide in next section.

\subsection{Equations of the schemes of higher rank  Higgs subbundles\label{eqns}} 

More generally, 
we study now the scheme $\hgrass d {\fE}$ of rank $d$ Higgs subbundles    of a rank $r$ Higgs bundle $\fE=(E,\phi) $  on an $n$-dimensional smooth variety $X$ for every $d$ in the range
$ 1 \le d \le r-1$. 
\begin{notation}\label{nota:concatenation}
    Given matrices of the same height $M_i\in \Mat_{a \times b_i}(R)$, for $i=1, \dots, s$ and for some ring $R$, we denote by $M_1|M_2|\cdots | M_s$ the matrix obtained by concatenation.  
\end{notation}

Let $U\subset X$ be a trivializing open subset; to simplify the notation we put $X=U$, i.e., we suppose that $E$ is trivial, with given global frame $\{e_i\in H^0(X,E)\ |\ i=1,\ldots,r \}$. The general case follows from the local analysis by standard gluing arguments. 

Consider the projective bundle $M_d=\mathbb P E^{\oplus d}$ whose points correspond, up to scalars, to nonzero $r\times d$ matrices with entries in $\mathscr{O}_X$. Denote by $\pi_j:E^{\oplus d}\to E$, for $j=1,\ldots,d$, the $j$-th canonical projection. We consider the trivializing frame $\{a_{ij}=\pi_j^* e_i\ |\ i=1,\ldots,r,\ \ j=1,\ldots,d \}$ on $E^{\oplus d}$, inducing homogeneous coordinates on $M_d$ over $X$. 

Let $U_d\subset M_d$ be the open locus corresponding to full-rank matrices.  In simple words, a point $[A]\in U_d$ corresponds to a matrix
\[
A=\begin{pmatrix}
    a_{11} &\cdots& a_{1d}\\
    \vdots &\cdots& \vdots\\
    a_{r1} &\cdots &a_{rd}
\end{pmatrix},
\]
whose columns define  a rank $d$ subbundle $F_A\subset E$. Viceversa, all subbundles arise in this way.\footnote{This association actually  expresses $U_d$ as   a Zariski locally trivial principal $\SL_d(\mathscr{O}_X)$-fibration over the Grassmann bundle $\grass{d}{E}$.} Let $\phi$ be the Higgs field, represented by  $r \times r$ matrices with entries $\phi_{ij}^{(h)}$,
 where $h=1,\dots,n$ and $i,j=1,\dots,r$. 
For a subbundle $F_A$ to be  Higgs invariant, one must have $\phi(F_A) \subseteq F_A\otimes\Omega^1_X$ which means that the matrices
\[
A^{(h)}=(A|\phi^{(h)}A ),
\]
for $h=1,\ldots,n$, all have rank $d$. This amounts to the vanishing of all the $(d+1)\times (d+1)$ minors of the matrices $A^{(h)}$ whose first $d$ columns come from $A$. Now, for every choice of a pair $(K,\alpha)$ with $K=(k_1<\cdots<k_{d+1})\subset \{1,\ldots,r\}$ of cardinality $d+1$ and $\alpha\in\{1,\ldots,d\}$, we put 
\begin{equation}\label{eq:eqdgrass}
     f_{K,\alpha}^{(h)}=\sum_{i=1}^r a_{i\alpha}  \sum_{u=1}^{d+1}(-1)^u \phi_{ui}^{(h)}p_{k_1\cdots\widehat{k_u}\cdots k_{d+1} } = 0,
\end{equation}
where
\[
\left(p_{\alpha_1\cdots\alpha_d}\ |\ 1\le\alpha_1<\cdots <\alpha_d\le r \right)
\]
are the Pl\"ucker   coordinates on the Grassmanniann.  Now, letting $(K,\alpha) $ and $h$ vary in the respective ranges, one  finds the generators of the ideal $J_{\fE,d}$ defining the closed subscheme of $U_d$  that parametrize frames of rank $d$ Higgs subbundles of $E$. To define the Higgs Grassmannian as a subscheme of the Grassmann bundle $\grass{d}{E}$ we have to compute the intersection
\begin{equation}\label{eq:interseciton}
    I_{\hgrass d {\fE}}=J_{\fE,d} \cap \mathscr{O}_X[a_{ij}]^{\SL_d(\mathscr{O}_X)} 
\end{equation}
where $\mathscr{O}_X[a_{ij}]^{\SL_d(\mathscr{O}_X)}$ is the invariant graded subring of $\mathscr{O}_X[a_{ij}]$ with respect to the natural action of $\SL_d(\mathscr{O}_X)$. 
This corresponds to identifying different frames whenever they span the same subbundle of $E$, see  
\cite[\S\,8.1]{Mukai-inv}
for more details. Recall that $\mathscr{O}_X[a_{ij}]^{\SL_d(\mathscr{O}_X)}$ is generated by the maximal minors of the matrix $A$ which we identify with the Pl\"ucker coordinates.

Note that if $d=1$ equation \eqref{eq:eqdgrass} reduces to equation \eqref{e:piphi}. This is a manifestation of the fact that $\mathbb P E $ is the classical Grassmanniann.

Let us denote by $d_{T,\beta}=\det A_{T,\beta}$, for $T=(t_1<\cdots<t_{d-1})\subset \{1,\ldots,r\}$ of cardinality $d-1$, and $\beta\in\{1,\ldots,d\}$, the determinant of the submatrix  

\[
A_{T,\beta}=\begin{pmatrix}
    a_{t_11} &\cdots&\widehat{a_{t_1\beta}}&\cdots & a_{t_1d}\\
    \vdots &\cdots&\vdots&\cdots & \vdots\\
    a_{t_{d-1}1} &\cdots &\widehat{a_{t_{d-1}\beta}}&\cdots &a_{t_{d-1}d}
\end{pmatrix},
\]
of $A$, consisting of the rows indexed by $T$ and with the $\beta$-th column removed. To compute the intersection in \eqref{eq:interseciton} we observe that the generators of $J_{\fE,d}$ are linear   polynomials in the variables $a_{ij}$'s and coefficients in the ring $\mathscr{O}_X[a_{ij}]^{\SL_d(\mathscr{O}_X)}$. Therefore, to compute \eqref{eq:interseciton} we use the relations coming from the Laplace expansion of the determinant. Precisely, for every $T,K\subset \{ 1,\ldots, r \}$ of respective cardinalities $d-1$ and $d+1$, and for every $h=1,\ldots,n$, we find an equation for the Higgs Grassmanniann:
\begin{equation} \label{eq:gKT}
\begin{aligned} 
g_{K,T}^{(h)} & = \sum_{\alpha=1}^d(-1)^\alpha d_{T,\alpha} f_{K,\alpha}
 \\
 & =\sum_{\alpha=1}^d(-1)^\alpha d_{T,\alpha}  \sum_{i=1}^r a_{i\alpha}  \sum_{u=1}^{d+1}(-1)^u \phi_{ui}^{(h)}p_{k_1\cdots\widehat{k_u}\cdots k_{d+1} }\\
 & =\sum_{u=1}^{d+1}(-1)^up_{k_1\cdots\widehat{k_u}\cdots k_{d+1} } \sum_{i=1}^r \phi_{ui}^{(h)}\sum_{\alpha=1}^d(-1)^\alpha d_{T,\alpha}  a_{i\alpha}  \\
 &=\sum_{u=1}^{d+1}(-1)^up_{k_1\cdots\widehat{k_u}\cdots k_{d+1} }\sum_{i=1}^r \phi_{ui}^{(h)} p_{T,i} =0  ,
\end{aligned}
\end{equation}
 with
 \[
     p_{T,i}=
 \begin{cases}
 0 &\mbox{ if }i\in T,\\
 (-1)^{e({T,i})}p_{T\cup \{i\}} &\mbox{ otherwise, }
 \end{cases}
 \]
 where $T\cup \{i\}$ is understood as a totally ordered set and $e({T,i})$ is the parity of the number of transpositions needed to   order the sequence $(i,t_1,\ldots,t_d)$.
This produces a, possibly redundant, set of $n\cdot\binom{r}{d-1}\cdot\binom{r}{d+1}$ generators for the ideal $I_{\hgrass d {\fE}}$ of the Higgs Grassmanniann. Note that this number is invariant under the substitution $d \to r-d$.  

 \begin{example}\label{eq23} 
We explain the construction just presented in a specific example, in  the case $n=1$, $r=3$ and $d=2$. We fix a Higgs bundle $\fE=(E,\phi )$ on $\BA^1$, with $E=\mathscr{O}_{\BA^1}^{\oplus 3}$,  equipped with the canonical frame $\{e_1,e_2,e_3\}$. Inside the open subset $U\subset \PP (E \oplus E) $ consisting of full rank matrices in $\Mat_{3\times 2}(\mathscr{O}_{\BA^1})$ up to homotheties by $\mathscr{O}_{\BA^1}^*$, the locus of Higgs frames is defined by the vanishing of the determinants 
\begingroup
\renewcommand*{\arraystretch}{1.5}
\begin{equation*}
   \det \begin{pmatrix}
        a_{11}& a_{12}&\sum_{i=1}^3\phi_{1i}a_{ij}\\
        a_{21}& a_{22}&\sum_{i=1}^3\phi_{2i}a_{ij}\\
        a_{31}& a_{32}&\sum_{i=1}^3\phi_{3i}a_{ij}\\
    \end{pmatrix}
\end{equation*} 
\endgroup
for $j=1,2$, where the $a_{ij}$'s are  homogeneous vertical  coordinates on $\mathbb P(E\oplus E)$, yielding 
\[
f_j=\sum_{i=1}^3\left[\phi_{1i} p_{23}-\phi_{2i} p_{13}+\phi_{3i}p_{12}\right]a_{ij}=0, \quad j=1,2. 
\]
Now to get the equations of $\hgrass2\fE$ inside $\grass2 E$  we  mod out   the action of $\SL_2(\mathscr O_{\mathbb A^1}) $. Algebraically, this consists in computing the invariant part of the ideal $(f_1,f_2)$. This has $1\cdot\binom{3}{1}\cdot\binom{3}{3}=3$ generators, namely
\[
g_1=a_{11}f_2-a_{12}f_1,\quad
g_2=a_{21}f_2-a_{22}f_1,\quad
g_3=a_{31}f_2-a_{32}f_1.  
\]
In terms of Pl\"ucker coordinates we have
\[
g_j=\sum_{i\not=j}  \left( \phi_{1i}p_{23}-\phi_{2i}p_{13}+\phi_{3i}p_{12}\right)p_{ji}
\]
for $j=1,2,3$, with the convention $p_{ji}=-p_{ij}$ for $j>i$.

Note that by putting 
\[
p_{12} = z_3,\quad
p_{13} = -z_2,\quad
p_{23} = z_1,
\]
after swapping the indices $i,j$, one finds the equations for $r=3$ and $d=1$ given in \eqref{e:piphi}.  This amounts to the fact that the action induced by $\phi$ on $E^\vee$ corresponds, in coordinates, to the transpose matrix $\phi^t$.
\end{example} 

One can see   that if all the matrices $\phi^{(h)}$ are the identity matrix, the equations \eqref{eq:eqdgrass} and \eqref{eq:gKT} reduce to the Pl\"ucker relations, which are identically satisfied on the Grassmann bundle. As a consequence we have:
\begin{lemma}\label{lemma:indepofdiag} If $\eta$ is a differential  1-form in $X$, the Higgs fields $\phi$ and $\phi - \id \otimes \eta$ define the same     subschemes of the Grassmann bundles of $E$.
\end{lemma}
\section{Structure of the Higgs Grassmannians on curves} 
\label{sec:curves}
Generally speaking, the Higgs Grassmannians have a very complicated structure. In this and in the following two sections we study some  cases.

We start with the Higgs Grassmannian $\hgrass 1 {\E}$ of a rank $r$ Higgs bundle $\E=(E, \phi)$ on a smooth curve $B$; we shall assume $B = \BA^1$, so that $\PP E$ is a trivial bundle, see \Cref{rmk:stratification}. Let $x$ be an affine coordinate on $\BA^1$ and  $z_1,\ldots,z_r$ homogeneous coordinates  on the fibers of $\PP E$, so that 
\begin{equation}\label{eq:BePE}
    B=\Spec \C[x] \quad \mbox{ and } \quad \PP E = \Proj_{B} \C[x,z_1,\ldots,z_r].
\end{equation} 

The Higgs endomorphism $\phi \in \Hom (E,E\otimes \Omega_B^1)$ is represented by a matrix of the form
\[
\phi(x)=\begin{pmatrix}
    \phi_{11} &  \cdots & \phi_{1r}\\
    \vdots &  \ddots & \vdots\\
    \phi_{r1} &  \cdots & \phi_{rr}\\
\end{pmatrix}\in  \Mat_{r \times r}
(\C[x]),
\]
and the Higgs Grassmannian $\hgrass 1 {\E} \subset \grass 1{E}$ according to equation \eqref{e:piphi} is defined by the ideal
\begin{equation}\label{eq:mingensIHGR}
    I_{\hgrass 1 {\E}}=\left(\left. \sum_{k=1}^r z_k(\phi_{jk}z_i-\phi_{ik}z_j) \ \   \right|\ \  1\le i<j\le r  \right).
\end{equation}  
We assume that the Higgs field has constant Jordan type, so that we can represent it, up to a base change, as a sequence of Jordan blocks 
\begin{equation}\label{eq:notBlocks}
    \phi= ( (\lambda_v , i_v )^{m_v} )_{v=1}^s,
\end{equation}
with 
\begin{itemize}
    \item $\lambda_v\in\C[x]$, for $v=1,\ldots,s$,
    \item $m_v,i_v\in\mathbb Z_{>0}$, for $v=1,\ldots,s$, such that $\sum_{v=1}^sm_vi_v=r$,
    \item $(\lambda_v,i_v)=(\lambda_u,i_u)$ if and only if $v=u$.
\end{itemize}

The symbol $(\lambda_v,i_v)^{m_v}$ encodes the eigenvalues $\lambda_v$, the sizes $i_v$ and the number of occurrences $m_v$ of the Jordan block of size $i_v$ and eigenvalue $\lambda_v$.

\begin{example}
    The sequence 
    \[
    \phi = ((\pi { x^2 +3x + 1 } ,1)^2,(3,2)^1,(i,3)^1)
    \]
    corresponds to the matrix 
    \[
    \phi=\begin{pmatrix}
        \pi {  x^2 +3x + 1 } & 0& 0& 0& 0& 0& 0\\[-0.5 em]
        0 & \pi {  x^2 +3x + 1 } & 0& 0& 0& 0& 0\\[-0.5 em]
        0 & 0  & 3& 1& 0& 0& 0\\[-0.5 em]
        0 & 0  & 0& 3& 0& 0& 0\\[-0.5 em]
        0 & 0  & 0& 0& i& 1& 0\\[-0.5 em]
        0 & 0  & 0& 0& 0& i& 1\\[-0.5 em]
        0 & 0  & 0& 0& 0& 0& i
    \end{pmatrix}{ \in \Mat_{7\times 7}(\mathbb C[x]) }. 
    \]
\end{example}

\begin{example}\label{ex:diagonal}(See also \Cref{finalregularthm}).
    The easiest example is the case when $\phi$ is diagonal:
    $$
    \phi= ((\lambda_1,1)^{m_1},\ldots,(\lambda_s,1)^{m_s}),
    $$
    i.e., a diagonal matrix with the first nonzero $m_1$ entries equal to $\lambda_1$, the next nonzero $m_2$ entries equal to $\lambda_2$, and so on.
    The ideal of the Higgs Grassmannian is   monomial, as one sees from \Cref{eq:mingensIHGR},  and  has the form
    \[
    \hgrass 1{E ,\phi}=\coprod_{v=1}^s B\times \PP \Lambda_v,
    \]
    where $\Lambda_v$ is the eigenspace relative to the eigenvalue $\lambda_v$, for $v=1,\ldots,s$.
\end{example}

The following Lemma describes explicitly the Higgs Grassmannian when $\phi$  consists of a single Jordan block.

\begin{lemma}\label{lemma:onejordanblock}
    Let $\phi=((\lambda,r)^1)$ be a single Jordan block. Then, the ideal of the Higgs Grassmannian is 
    \begin{equation}\label{eq:idealoneblock}
        I_{\hgrass 1{E,\phi}} = \left( \rk\begin{pmatrix}
        z_1&\cdots & z_{r-1}\\z_2&\cdots &z_r
    \end{pmatrix}\le 1 ,\; z_iz_r \ |\ i=2,\ldots, r \right).
    \end{equation} 
   As a consequence, each fiber of $\rho_1 \colon \hgrass 1{\E} \to B
$ is a curvilinear zero-dimensional scheme of degree $r$, topologically consisting of a single point.
\end{lemma}

\begin{proof}
    Thanks to \Cref{lemma:indepofdiag} we can assume $\lambda=0$. We consider the minimal generating set of $I _{\hgrass 1 {E,\phi}} $ given in \eqref{eq:mingensIHGR} and we partition it according to the following criterion: 
\begin{enumerate}[itemsep=-2pt, label=\textit{(\Alph*)}, ref=\textit{(\Alph*)}]
    \item \label{it:A} $1\le i<j = r$,
    \item \label{it:B} $1\le i < j =i+1<r$,
    \item \label{it:C} $1\le i< i+1< j < r$.
\end{enumerate} 

     On the other hand, a minimal generating set for the right-hand side of equation \eqref{eq:idealoneblock} is given by 
\begin{enumerate}[itemsep=-2pt, label=\textit{(\alph*)}, ref=\textit{(\alph*)}]
    \item \label{it:a} $z_{i}z_r$ for $i=2,\ldots,r$,
    \item \label{it:b} $z_{i+1}^2-z_{i}z_{i+2}$, for $i=1,\ldots,r-2$,
    \item \label{it:c} $z_{j}z_{i+1}-z_{i}z_{j+1}$, for  $1\le i < i+1< j < r$.
\end{enumerate} 
    
    It is immediate to see   that the generators of type \ref{it:a}, \ref{it:b} and \ref{it:c} correspond bijectively to the minimal generators of type \ref{it:A}, \ref{it:B} and \ref{it:C}, respectively. This gives the equality in \eqref{eq:idealoneblock}.

    As for the second part of the statement,  denote by $J_i$, for $i=1,2$, the ideals
    $$
    \begin{aligned}
        J_1 &=  \left( \rk\begin{pmatrix}
        z_1&\cdots & z_{r-1}\\z_2&\cdots &z_r
    \end{pmatrix}\le 1 \right), \qquad
        J_2 &=  \left(z_iz_r \ |\ i=2,\ldots, r  \right),
    \end{aligned}
    $$
  and by $Z_i\times B\subset \PP^{r-1}_B$, for $i=1,2$,  the associated closed subschemes. Note that for every $b\in B$ the fiber $ \rho_1^{-1}(b)$ is isomorphic to the scheme $Z=Z_1\cap Z_2\subset \PP^{r-1}$ and we have
    \[
    \hgrass 1 {\fE}\cong Z\times B.
    \]

    The scheme $Z_1 $ is a rational normal curve of degree $r-1$. Therefore, the scheme $Z $ is curvilinear. Moreover,
    \[\sqrt{I_{\hgrass 1 {E,\phi}} }=\left(z_2,\ldots,z_r\right)\]
    which implies that $Z$ is a fat point supported at $[1:0:\cdots:0]$. It remains to compute its degree.
The strategy is to consider the  restriction of the Veronese embedding of $Z_1$ 
    \[
    \begin{tikzcd}[row sep =tiny]
        \BA^1\arrow[r,"v_{1,{r-1}}"]&\PP^{r-1}\\
        t:\arrow[r,mapsto]& {[1:t:t^2\cdots :t^{r-2}:t^{r-1} ]}
    \end{tikzcd}
    \]
    {  to an open chart, }and to compute $v_{1,r-1}^*Z_2$. Explicitly, the ideal of  $v_{1,r-1}^*Z_2$ is
    \[
    \left(t^{i}t^{r-1} \ |\ i=1,\ldots, r-1  \right)=(t^r)\subset \C[t].
    \]
    This gives $\deg Z= r$.
\end{proof}

\begin{example} \label{ex:r=3}
For $r=3$ the ideal is generated by $z_1z_3-z_2^2$, $z_2z_3$ and $z_3^2$. Its radical is the ideal $(z_2, z_3)$.
\end{example}\label{example:r=3}
  \Cref{prop:idealoneeigen} and \Cref{prop:unautoval} will study the case  when $\phi$ has several Jordan blocks, all having the same eigenvalue. We make some preparations and introduce some notation.
\begin{notation}\label{notation:SV}
    Let $R$ be  a ring. Consider a polynomial ring $R[z_1,\ldots,z_r]$  in $r$ variables,  integers $s\ge 1$ and $d_1,\ldots,d_s\ge 1$ such that $\sum_{i=1}^sd_i\le r$. Fix a set 
    \[
    \mathscr{A}=\left\{A_1,\ldots,A_s \right\}
    \]
    of $s$ disjoint subsets  of $\{1,\ldots,r\}$ of respective cardinalities $d_i$, for $i=1,\ldots,s$.

Now, given $\mathscr{A}$ as above, define the matrix $M_{\mathscr A}$ as
\[
M_{\mathscr A}=M_{A_1}|M_{A_2}|\cdots|M_{A_s},
\]
where, if $A_i=\{i_1<\ldots<i_{d_i}\}$, the matrix $M_{A_i}\in\Mat_{2 \times (d_i-1)} (R[z_1,\ldots,z_r])$ is
\[
M_{A_i}=\begin{pmatrix}
    z_{i_1}& z_{i_2}&\cdots &z_{i_{d_i-1}}\\
    z_{i_2}& z_{i_3}&\cdots &z_{i_{d_i}}\\
\end{pmatrix}
\]
(for $d=1$ this is taken as the empty matrix).
\end{notation}
\begin{defin}
    The scheme $\mathsf{SV}_{\mathscr A}$ is the vanishing locus of the ideal
\[
I_{\mathscr A}=\left(\rk M_{\mathscr A}\le 1\right).
\]
\end{defin}

 We now state a technical lemma that will be used in the proof of     \Cref{prop:idealoneeigen} and \Cref{prop:unautoval}.
 
\begin{lemma}\label{lemma:verosegre}
   In the setting of \Cref{notation:SV}, with $R=\mathbb C$, assume that $d_i=d_j=d$ for all $i,j=1,\ldots,s$. Then, the subscheme $\mathsf{SV}_{\mathscr A}\subset \PP^{r-1}$ is a cone with vertex\footnote{Our convention in the case $r=ds$ is $\PP^{-1}=\emptyset$, i.e.  $\iota $ is the identity and $\mathsf{SV}_{\mathscr A}$ is smooth.} $\PP^{r-ds-1}$ over the image of the embedding $\mathsf{sv}_{\mathscr A} \colon \PP^{1} \times \PP^{s-1} \to \PP^{r-1}$ given by the composition
   \[
   \begin{tikzcd}[column sep = large]
      \PP^1\times \PP^{s-1} \arrow[rrr,bend left=15,"\mathsf{sv}_{\mathscr A}"]\arrow[r,"{(v_{1,{d}-1},\id )}"']&\PP^{d-1}\times \PP^{s-1}\arrow[r,"{s_{d-1,s-1}}"'] &\PP^{ds-1} \arrow[r,"\iota"'] &\PP^{r-1}  
   \end{tikzcd},
   \]
   where $v_{1,{d}-1}$ is the $({d}-1)$-th Veronese embedding of $\PP^1$, $s_{d-1,s-1} $ is the Segre embedding and $\iota$ is the linear embedding of the subspace defined by the ideal $\left(z_a\ |\ a\in \{1,\ldots,r\}\setminus \mathscr{ A}\right)$.  
\end{lemma}
\begin{proof}
    Without loss of generality we may assume 
    \[
    A_i=\{d(i-1)+1,d(i-1)+2,\ldots,d(i-1)+d\}
    \]  
    for $i=1,\ldots,s$. Then, the morphism $\mathsf{sv}_{\mathscr A}$ takes the form 

     \begin{equation*}
         \begin{aligned}
              \PP^1\times \PP^{s-1} 
\xrightarrow{\hspace{.5cm}\mathsf{sv}_{\mathscr A}\hspace{.5cm}} 
&\hspace{.2cm}\PP^{r-1}  \\
              {([u_0:u_1],[v_1:\cdots: v_s])}  \xmapsto{\hspace{.2cm}} &
             [ v_1u_0^{d-1} :  v_1u_0^{d-2}u_1:  \cdots :  v_1u_1^{d-1}: \\ &\hspace{1cm}  v_2u_0^{d-1} :  v_2u_0^{d-2}u_1:  \cdots :  v_2u_1^{d-1}: \\ &\hspace{2cm} \cdots :v_su_0^{d-1} :  v_su_0^{d-2}u_1:  \cdots :  v_su_1^{d-1} :0\cdots :0]
    \end{aligned}
     \end{equation*}
    so that the conclusion follows from standard arguments in classical algebraic geometry.
\end{proof}
\begin{remark}\label{rem:bijref}
    When $s=1$ and $d=r$ the variety $\mathsf{SV}_{\mathscr A}$ is a rational normal curve of degree $r-1$. This   already appeared in \Cref{lemma:onejordanblock}. Note also that, given an injection $\theta:\mathscr{A}\to \mathscr{A}'$ such that $A\subset \theta (A) $, for all $A\in\mathscr{A}$, we have $\mathsf{SV}_{\mathscr A'}\subset \mathsf{SV}_{\mathscr A}$.
\end{remark}

The next Proposition will be instrumental in proving \Cref{prop:unautoval}.

\begin{prop}\label{prop:idealoneeigen}
{Let $B$ and $E$ be as in \eqref{eq:BePE}.}    Fix an increasing sequence of positive integers $0< i_s< i_{s-1}< \cdots <i_1$ and positive integers $m_1,\ldots ,m_s > 0$ for some $s\ge 1$.  Consider the Higgs field 
\begin{equation}\label{eq:onelambda}
\phi=((\lambda,{i_v})^{m_v})_{v=1}^s
\end{equation}
for some section $\lambda\in H^0(B,\mathscr{O}_B)$. Then, the ideal of the Higgs Grassmannian is
\begin{equation}\label{eq:equaoneblocktutta}
     I_{\hgrass 1 {\fE}} = \left(z_bz_c \ | \ b\in \mathscr B,\ c\in \mathscr C\right) + I_{\mathscr {A}},
\end{equation}
where
\begin{equation}\label{eq:A}
\mathscr A= \left\{\Big\{ ki_v+j+ \sum_{t=1}^{v-1}m_ti_t  \; \Big| \; j=1,\ldots, i_{v} \Big\} \; \big| \; v=1,\ldots,s,\ k=0,\ldots,m_v-1 \right\},
\end{equation}
$\mathscr{B}={ \bigcup_{A\in\mathcal A}}  (A\setminus \min A )$, $\mathscr C= \{ \max A \; | \; A\in\mathscr A\}$, and $I_{\mathscr {A}}$ is the ideal of ${\mathsf{SV}_{\mathscr {A}}} $.

In particular, the Higgs Grassmannian is irreducible of dimension $-1+\sum_{v=1}^sm_v$. 
\end{prop}

\begin{example}\label{rem:A} 
Let us spell out the explicit form of the set $\mathscr A$ in \Cref{eq:A} in a specific example.
Let $\phi=((0,3)^2, (0,2)^2, (0,1)^2)$, so that  $i_1=3, i_2=2, i_3=1$ and $m_i=2$ for all $i$. Then:
\begin{itemize}
    \item Group 1 ($v=1$): $j=1, 2, 3$, size = 3, $\sum_{t=1}^{0}m_ti_t=0$, $3k+j$ $\implies \{1, 2, 3\}, \{4, 5, 6\}$
    \item Group 2 ($v=2$): $j=1, 2$, size = 2, $\sum_{t=1}^{1}m_ti_t=6$, $2k+j+6$ $\implies \{7, 8\}, \{9, 10\}$
    \item Group 3 ($v=3$): $ j=1$, size = 1, $\sum_{t=1}^2m_ti_t=10$,  $k+j+10$ $ \implies \{11\}, \{12\}$
\end{itemize}
 so that
\[
\mathscr{A} = \big\{ \{1, 2, 3\}, \{4, 5, 6\}, \{7, 8\}, \{9, 10\}, \{11\}, \{12\} \big\},
\] and 
\[
\mathscr{B} = \{ 2, 3,  5, 6, 8,  10\}, 
\quad \mathscr{C} = \{3, 6, 8, 10, 11, 12\}.
\]
\end{example}

\begin{proof}[Proof of \Cref{prop:idealoneeigen}]
 Thanks to  \Cref{lemma:indepofdiag} we may assume $\lambda=0$.  Note  that $r=\sum_{v=1}^sm_vi_v$. To prove \eqref{eq:equaoneblocktutta} we proceed similarly to the proof of  \Cref{lemma:onejordanblock}. According to \eqref{eq:mingensIHGR}, for $1\le i <j\le r$, there are  three  types of generators of the ideal $I_{\hgrass1 {\fE}}$: 
\begin{itemize}
        \item if $i\in \mathscr{C}$ and $j\in \mathscr{B}$ we have $\phi_{ik}=0$ for all $k=1,\ldots,r$. The unique summand that survives in the  $(ij)$-generator of \eqref{eq:mingensIHGR} is $z_{i}z_j$. This corresponds to a minimal generator of the first summand in \eqref{eq:equaoneblocktutta};
        \item  if $i,j\notin \mathscr{C}$, then only two summands of the  $(ij)$-generator of \eqref{eq:mingensIHGR} survive. Their sum gives the generator  $z_{j+1}z_i-z_{i+1}z_j$, for $\{i,j\}\subset \{1,\ldots,r\}\setminus \mathscr{C}$. This yields all the generators of the second summand of the right-hand side in \eqref{eq:equaoneblocktutta};
        \item all  remaining choices of $1\le i<j \le r$ produce the polynomial 0.
\end{itemize}
This concludes the first part of the proof.

The second part of the statement follows directly from the first. Indeed a straightforward  computation shows that
\[
\sqrt{I_{\hgrass 1 {\fE}}}= (z_a\ |\ a \in  \mathscr{B} ),
\]
which implies $\hgrass 1 {\fE}_{\red}\cong B \times \PP^{-1+\sum_{v=1}^s m_v}$, because $\sum_{v=1}^s m_v=|\mathscr A|=|\{1,\ldots,r\}\setminus \mathscr{B}|$.
\end{proof}
 
\begin{notation}\label{notation:SETS} Fix the same notation as in   \Cref{prop:idealoneeigen}, and denote by $r\in\Z_{>0}$ the sum $r=\sum_{v=1}^si_v$. For for $v=1,\ldots,s$, we consider the following subsets of $\{1,\ldots,r\}$  
\[
\begin{aligned}
    \mathscr{A}_{v}&=\left\{ \Big\{ki_\alpha+j+\sum_{t=1}^{\alpha-1} m_t i_t  \; \big| \; j=1,\ldots,i_v \Big\} \Big| \;
    \alpha=1 \ldots,v \; , \; k=0 \ldots,m_\alpha-1 
\right\}   ,   \\
    \mathscr{L}_v&= \left\{1,\ldots,r \right\}\setminus   \bigcup_{A\in\mathscr {A}_v}A , \\
    \mathscr {B}_{v}&=\bigcup_{A\in\mathscr{A}_v}A \setminus \min A    ,\\
     \mathscr {C}_v & = \left\{\max A\ |\ A\in\mathscr{A}_v \right\} .
\end{aligned}
\] 
\end{notation} 

\begin{example}
Now, we evaluate the set $\mathscr A_v $ for $\phi$ as in \Cref{rem:A}, i.e., when 
$s=3$, $i_1 = 3$, $i_2 = 2$, $i_3 = 1$, and $m_1 = 2$, $m_2 = 2$, $m_3 = 2$. Since $m_i=2$ the index $k$ takes values  $ 0, 1$.

In $\mathscr{A}_1$, the inner loop index $j$ runs from $1$ to $i_1 = 3$, every inner set will have exactly 3 elements, and the outer loop runs only on $\alpha=1$.

For $\alpha=1$ ($i_1 = 3$, $m_1 = 2$), the   sum is $\sum_{t=1}^{0} m_t i_t = 0$, and $ki_\alpha+j$ is $3k  + j$, so that
\begin{itemize}
    \item for $k=0$ we get  $\{1, 2, 3\}$,
    \item for $k=1$ we get  $ \{4, 5, 6\}$,
\end{itemize}
and
\[
\mathscr{A}_1 = \big\{ \{1, 2, 3\}, \{4, 5, 6\} \big\}.
\] For the other sets one has 
\[
\mathscr{L}_1 =  \{7, 8, 9, 10, 11, 12\}, \quad  \mathscr{B}_1 =\{ 2, 3, 5, 6\}, \quad \mathscr{C}_1 = \{ 3, 6\}.
\]  

For $\mathscr{A}_2$, the inner loop index $j$ runs from $1$ to $i_2 = 2$. Every inner set  has 2 elements, and the outer loop runs on $\alpha=1$ and $\alpha=2$. For $\alpha=1$ ($i_1 = 3$, $m_1 = 2$), the  sum is $0$, while $ki_\alpha+j$ is $3k + j$ with $j \in \{1, 2\}$; then
\begin{itemize}
    \item $k=0$ yields $  \{1, 2\}$,
    \item   $k=1$ yields $\{4, 5\}$.
\end{itemize}

For $\alpha=2$ ($i_2 = 2$, $m_2 = 2$)
the  sum is $m_1 i_1 = 2 \cdot 3 = 6$, while $ki_\alpha+j$ is $2k + j + 6$, so that 
\begin{itemize}
    \item   $k=0$ gives $ \{7, 8\}$,
    \item $k=1$ gives  $ \{9, 10\}$,
\end{itemize} 
and
\[
\mathscr{A}_2 = \big\{ \{1, 2\}, \{4, 5\}, \{7, 8\}, \{9, 10\} \big\}
\] which gives all   blocks of size 2 of the matrix representing the Higgs field $\phi$.

Therefore, 
\[ \mathscr{L}_2 =\{3, 6, 11,12\}\quad \text{and}\quad \mathscr{B}_3 = \{2, 5, 8, 10\}= \mathscr{C}_3.\]

In $\mathscr{A}_3$,  $j$ runs from $1$ to $i_3 = 1$, every inner set has 1 element, and the outer loop runs on $\alpha=1, 2, 3$. For $\alpha=1$, $3k + 1$ gives $\{1\}, \{4\}$, and for $\alpha=2$, we have $2k + 1 + 6 = 2k + 7$ which yields $ \{7\}, \{9\}$. Finally, for $\alpha=3$, $i_3 = 1$, $m_3 = 2$, the sum is $m_1 i_1 + m_2 i_2 = 6 + 4 = 10$, and  $ki_\alpha+j$ is  $  k + 11$ so that we obtain $\{11\}, \{12\}$.
Eventually,
\[
\mathscr{A}_3 = \big\{ \{1\}, \{4\}, \{7\}, \{9\}, \{11\}, \{12\} \big\}
\] which corresponds to all the eigenvectors of $\phi$.

Moreover, 
\[ \mathscr{L}_3 = \{1, \dots, 12\} \setminus \{1,4,7,9,11,12\} =\{2, 3, 5, 6, 8, 10\}, \quad \mathscr{B}_3 = \emptyset, \quad \mathscr{C}_3 = \{1, 4, 7, 9, 11, 12\}. \]
\end{example}

\begin{remark}
    Notice that the sets $\mathscr {A}_v$, for $v=1,\ldots,s$, in  \Cref{notation:SETS} satisfy the requirements in  \Cref{lemma:verosegre}. Indeed,  each  $ A\in \mathscr{A}_v$ is a subset of $\{1,\ldots,r\}$ of cardinality $i_{v}$. Moreover, these subsets are all disjoint from each other.
\end{remark}

The next Theorem describes the schematic structure of the Higgs Grassmannian of  \Cref{prop:idealoneeigen}. We remind that we are considering a trivial Higgs bundle over the affine line as in \eqref{eq:BePE}, with Higgs field as in \Cref{eq:onelambda}.
  
\begin{thm}\label{prop:unautoval}
Fix the same notation as in  \Cref{prop:idealoneeigen}. The fibers $Z_{\phi}$ of $$\rho_1 \colon \hgrass{1}{E,\phi} \to B$$ stratify into $s$ embedded (irreducible) components $ V_1, \ldots , V_s$ such that
\begin{itemize}
\item the ideal of $V_v$ is \begin{equation}\label{eq:IVv}
        I_{V_v} =   \left(z_j \ |\ j\in \mathscr{L}_v \right) + \left(z_bz_c \ | \ b\in \mathscr{B}_v,\ c\in \mathscr{C}_v \right) + I_{\mathscr {A}_v},
    \end{equation}
    where $\mathscr {A}_v, \mathscr{B}_v,\mathscr{C}_v,\mathscr{L}_v$ are defined as in \Cref{notation:SETS}, and  
    $I_{\mathscr {A}_v}$ is the ideal of ${\mathsf{SV}_{\mathscr {A}_v}}$;
\item $\dim V_v = d_v = -1+ \sum_{i=1}^vm_i $, and the reduction of every $V_v$ is isomorphic to  $  \P^{d_v} $;
\item $V_v$ admits a structure of $Z$-bundle over $  \P^{d_v} $, where $Z$ is a curvilinear zero-dimensional fat point of degree $i_v$.
\end{itemize}
\end{thm}
\begin{proof}
    Again, thanks to   \Cref{lemma:indepofdiag} we can assume $\lambda=0$.  Let us denote by $r=\sum_{v=1}^sm_vi_v$ the size of $\phi$. First, notice that none of the ideals involved in the right-hand side of \eqref{eq:IVv} is redundant because, by construction, there is no inclusion   $\mathscr{L}_v\subset \mathscr{L}_{v'}$ for $v\not=v'$.  Now we show by induction on $s\ge 1$ that
   \begin{equation}\label{eq:primaryDec}
         I_{{\phi}}=\bigcap_{v=1}^s I_{V_v},
   \end{equation}
    where $I_\phi$ is the ideal of (any) fiber $Z_{\phi}\subset\P^{r-1}$, is a primary decomposition of $I_\phi$.  
    The ideal $I_\phi$ is  generated by the same polynomials in \eqref{eq:equaoneblocktutta} seen as polynomials in $\C[z_1,\ldots,z_r]$ instead of $\C[x,z_1,\ldots,z_r]$.

In the case $s=1$, equation \eqref{eq:primaryDec} is true by \Cref{lemma:onejordanblock}. We move now to the inductive step.  Denote by $I_{s}$ the ideal $I_{s}=\bigcap_{v=1}^{s-1} I_{V_v}$. In particular we have $\bigcap_{v=1}^{s} I_{V_v} =I_s\cap  I_{V_{s}}$. On the other hand, we have by construction that the set of indices $\bar {\mathscr{L}}=\bigcap_{v=1}^{s-1} \mathscr{L}_{v}$
corresponds to the ideal $(z_{r-i_{s}m_s+1},\dots,z_{r-1},z_r)$.
 By the  inductive hypothesis this yields the decomposition
\[
I_{s}= \left(z_j \ |\ j\in \bar{\mathscr{L}} \right) +  \bar I,
\]
where $\bar I$ is the ideal of the Higgs Grassmannian of the rank 
$r- i_sm_s$ trivial bundle $E'$ with Higgs field
$\phi_s = ((\lambda,{i_v})^{m_v})_{v=1}^{s-1}$, and we are thinking of $\PP E'$ as a subbundle of $\PP E$.
By    \Cref{prop:idealoneeigen} we then get 
\begin{equation}\label{eq:Is}
    I_{s}= \left(z_j \ |\ j\in \bar{\mathscr{L}} \right) + \left(z_bz_c \ | \ b\in \bar{\mathscr{B}},\ c\in \bar{\mathscr{C}} \right) + I_{{\bar{\mathscr {A}}}},
\end{equation}
where $\bar{\mathscr{A}},\bar{\mathscr{B}}$ and $\bar{\mathscr{C}}$ are defined as in   \Cref{prop:idealoneeigen}. Denote the first two summands in \eqref{eq:Is} by $U,V$ and denote by $X,Y$ the first two summands of $I_{V_s}$ in \eqref{eq:IVv}{ for $v=s$}. 
We have
\[
I_\phi\subset I_s\cap I_{V_s}
\]
because

\[
\det \begin{pmatrix}
    z_i&z_j\\
    z_{i+1}&z_{j+1}
\end{pmatrix} \in (X+I_{{\mathscr{A}}_s})\cap (U+I_{\bar{\mathscr{A}}}), \quad z_bz_c  \in (X+Y)\cap (U+V)
\]
for $1\le i < j < r$,   $b\in\mathscr{B}$ and $c\in\mathscr{C}$.

For $1\le k\le i_1$ denote by $\mathscr{B}^{(k)}$ and $\mathscr{C}^{(k)}$ the   sets
\[
\mathscr{B}^{(k)} = \bigcup_{\{ a_1<\ldots <a_u \}\in \mathscr{A}} \{a_{1+k},\ldots,a_u \}  
\quad 
 \text{and}
\quad
\mathscr{C}^{(k)} = \bigcup_{\{ a_1<\ldots <a_u \}\in \mathscr{A} } \{a_{u},\ldots,a_{u-k+1} \}, 
\]
with the convention $\{a_{1+k},\ldots,a_u \}=\varnothing$ for $1+k>u$. Now we write $I_\phi$ using a redundant\footnote{The first summand contains the others and it equals the first summand in \eqref{eq:equaoneblocktutta}.} set of monomial generators,
\begin{equation}
    \label{eq:Iphiextend}
    I_\phi = \sum_{k=1}^{i_1}(z_{b }z_{c } \ |\, b\in\mathscr{B}^{(k)},\,c\in\mathscr{C}^{(k)}  ) + I_{\mathscr{A}}.
\end{equation}
This simplifies the rest of the proof, as, by construction, all the monomials of degree 2 in $I_{\phi}$ are displayed in \eqref{eq:Iphiextend}. Now, notice that  the   equalities 
    \[
    I_s = U+V+I_{\mathscr{A}},\quad \quad I_{V_s} = X+Y+I_{\mathscr{A}}
    \]
    hold. 
Let us denote by $\Gamma$  the first summand in \eqref{eq:Iphiextend}. We use these decompositions to list all the monomials in $I_s    \cap  I_{V_s}$, then the check that they belong to $\Gamma$ will be immediate, and this will conclude the proof. The monomial generators in the intersection $I_s\cap I_{V_s}$ are all of the  form $z_{b}z_c$ for:
    \begin{itemize}
        \item $b\in \bar{\mathscr{L}}$, $c\in \mathscr{L}_s$, and hence belong to $(z_{b }z_{c } \ |\, b\in\mathscr{B}^{(i_s)},\,c\in\mathscr{C}^{(i_s)}  ) \subset \Gamma $, or,
        \item $z_{b}z_c\in V $, and hence belong to $X$ because $V\subset X$ since $\bar{\mathscr{C}}\subset \mathscr{L}_s$, or,
        \item $b\in \mathscr{B}_s $, $c\in \mathscr{C}_s\cap \bar{\mathscr{L}}$, and hence belong to $(z_{b }z_{c } \ |\, b\in\mathscr{B}^{(i_s)},\,c\in\mathscr{C}^{(i_s)}  ) \subset \Gamma $, or,
        \item $b\in \mathscr{B}_s' $, $c\in \mathscr{C}_s$, where
        \[
        \mathscr{B}_s'=  \bigcup_{\{ a_1<\ldots <a_u \}\in \mathscr{A}} \{a_{1+u-i_s},\ldots,a_{i_s} \}   ,
        \] 
        and we adopt the same notation as above for the empty limit cases.
        \item All the other monomials are obtained by using the binomials in $I_{\mathscr{A}}$ and  are recovered via \eqref{eq:Iphiextend}.
    \end{itemize}
Clearly all these generators belong to $\Gamma$, and this concludes the proof.
\end{proof}
\begin{example}
    
    Put $s=2$, $i_1=4,i_2=2$, and consider $\phi = ((0,4),(0,2))$, i.e. the matrix 
    \[
    \phi =\begin{pmatrix}
        0 & 1 & 0 & 0 & 0 & 0 \\[-0.5em]
        0 & 0 & 1 & 0 & 0 & 0 \\[-0.5em]
        0 & 0 & 0 & 1 & 0 & 0 \\[-0.5em]
        0 & 0 & 0 & 0 & 0 & 0 \\[-0.5em]
        0 & 0 & 0 & 0 & 0 & 1 \\[-0.5em]
        0 & 0 & 0 & 0 & 0 & 0  
    \end{pmatrix}.
    \]
    Then, we have
 \begin{itemize}
     \item[]   $\mathscr{A}_{1}=\{\{1,2,3,4\}\}  $,\ 
     $\mathscr B_1=\{2,3,4\}$,\  $\mathscr C_1=\{4\}$,\ 
     $\mathscr L_1 = \{5,6\}$; 
     \item[]   $\mathscr{A}_{2}=\{\{5,6\} ,\{1,2\}\}$,\ 
     $\mathscr B_2=\{2,6\}$,\  $\mathscr C_2=\{2,6\}$,\ 
     $\mathscr L_2 = \{3,4\}$.  
 \end{itemize}
 This gives
 \[ 
 I_{V_1}= \left(z_5,z_6\right)+\left(z_4^2,z_2z_4,z_3z_4\right)+\left(\rk\begin{pmatrix}
     z_1 &z_2&z_3\\
     z_2&z_3&z_4
 \end{pmatrix}\le 1\right),
 \]
 and 
 \[
 I_{V_2}= \left(z_3,z_4\right)+\left(z_2^2,z_2z_6,z_6^2\right)+\left(z_1z_6-z_2z_5\right).
 \]
\end{example}

We now  consider the case with multiplicities, i.e.,   the $m_{i}$'s in \eqref{eq:notBlocks} are not necessarily all equal to 1.

\begin{example}
We presented a first example with multiplicities in   \Cref{ex:diagonal}. Now we see an example with a unique eigenvalue and bigger blocks. We consider the matrix in Jordan form
\[
\phi=((0,2)^2).
\]
    This means that  $s=1$, $i_1=2$, and $m_1=2$, and the  matrix $\phi$ has the form
    \[
    \phi =\begin{pmatrix}
        0 & 1 & 0 & 0   \\[-0.5em]
        0 & 0 & 0 & 0   \\[-0.5em]
        0 & 0 & 0 & 1   \\[-0.5em]
        0 & 0 & 0 & 0     
    \end{pmatrix}.
    \]
    We have   $\mathscr{A}_{1}=\{\{1,2\},\{3,4\}\}  $, and  
 
 \[
 I_{V_1}=0+ \left(z_2^2,z_2z_4, z_4^2\right) + \left( z_2z_3-z_1z_4 \right).
 \] 
\end{example}

The following result together with \Cref{prop:unautoval} concludes the classification of the Higgs bundles in dimension 1 when the Jordan type of the Higgs field is constant.

\begin{thm}\label{thm:maincurvecostante}
    Let $\fE=(E,\phi)$ be a Higgs bundle on a smooth  curve $B$, and let $U$ be an open subset of $B$ where the Jordan type of $\phi$ is constant, given by
     \[((\lambda_v,i_{vj})^{m_{vj}}\ |\ 1\le v \le s,\ 1\le j\le j_v),
    \]
     with $\lambda_v\not=\lambda_u$ for $v\not=u$. 
    Denote by $\fE_v=(E_v,\phi_v)$, with  $E_v\subset E_{|U}$ and $\phi_v=\phi_{|E_v}$,  the invariant subbundle relative to the eigenvalue $\lambda_v{ \in H^0(U,\cO_U)}$. Then, the decomposition of the Higgs Grassmannian   $\hgrass{1}{\fE_{|U}}$  in connected components is
    \begin{equation}\label{eq:statement}
    \hgrass {1}{\fE_{|U}}= \coprod_{v=1}^s \hgrass{1}{\fE_v}.        
    \end{equation} 
\end{thm} 
\begin{proof} 
Let us denote by $Y$ the right-hand side of \eqref{eq:statement}. By the  universal property of $\hgrass {1}{\fE_{|U}}$, see   \Cref{rem:universe}, we have a canonical morphism of schemes
 \[
\begin{tikzcd}
     Y \arrow[rr, "\Pi"]\arrow[dr]&& \hgrass{1}{\fE_{|U}}\arrow[dl] \\
     & B
\end{tikzcd}
\]
induced by the inclusions of Higgs bundles $\fE_v\hookrightarrow \fE_{|U}$ for $v=1,\ldots,s$. Note that the scheme $Y$ represents the functor
\[
    \begin{tikzcd}[row sep =tiny]
        \text{Sch}^{\mathrm{op}}  \arrow[r,"G"]&   \text{Sets}\\
        T\arrow[r,mapsto] & \{ \text{rank $1$ Higgs subbundles of $p_2^\ast \fE_v$, for some $v=1,\ldots,s$}\},
    \end{tikzcd}
\]
where $p_2\colon T\times X \to X$ is the projection. 
In particular, the morphism of schemes $\Pi$ induces a natural transformation $\alpha: G \Rightarrow F_{1,\fE_{|U}}$, see \Cref{eq:functor} for the definition of $F_{1,\fE_{|U}}$. Explicitly, for every scheme $T$  the morphism $\alpha_T: G(T)\to F_{1,\fE_{|U}}(T) $ regards a rank 1 subbundle $\mathfrak L\subset p_2^*\fE_v$, for some $v=1,\ldots,s$, as a rank 1 subbundle $\mathfrak L\subset p_2^*\fE_{|U}$.

The conclusion follows from the fact that for every scheme $T$, and for every $p_2^*\phi$-invariant rank 1 subbundle $\mathfrak L\subset p_2^\ast\fE_{|U}$, there is a morphism $\varpi:T \to Y$ such that $\mathfrak L = \varpi^\ast \widetilde\gS$, where $\widetilde\gS$  is the universal rank 1 subbundle on $Y$, i.e.~the sheaf agreeing on each component of $Y$ with the sheaf $\gS_{1,v}$ introduced in \eqref{eq:uniSeq}. This  means that $Y$ represents the functor $F_{1,\fE}$, and hence that $\Pi $ is an isomorphism.  
\end{proof}

 The following Corollary is a kind of qualitative summary of what we have so far understood about the structure of the Higgs Grassmannian.

\begin{corol} \label{cor:cor}    With the same notation and hypotheses of   \Cref{thm:maincurvecostante}, the morphism $\hgrass{1}{\fE}\to B$ is finite if and only if the eigenvalues corresponding to the Jordan decomposition of $\phi$ are all distinct. Moreover, the Higgs Grassmannian is reduced if and only if $\phi$ is diagonalizable.
\end{corol}

\begin{proof}
    Direct consequence of   \Cref{thm:maincurvecostante} and \Cref{prop:unautoval}.
\end{proof}

\begin{remark}\label{rmk:stratification}
When the base curve is irreducible, by semicontinuity the Jordan type of the Higgs field is constant on an open dense subset, but on finitely many points the Jordan type may be finer. However, by analyzing the Higgs Grassmannian on each stratum where the Jordan type is constant, one can reconstruct its global structure. In the following we make a few examples.
Actually, in the rank two case  a more intrinsic analysis is feasible, as we shall see in       \Cref{section:discriminant}.
\end{remark}

\begin{example} \label{ex:J1}
   Let $B=\BA^1$ with an affine coordinate $x$, and let $E$ be the trivial rank 2 bundle with the Higgs field $$\begin{pmatrix}x&0\\1&x
       \end{pmatrix}.$$ 
   Here the Jordan type is constant, with a single Jordan block. According to \Cref{lemma:onejordanblock} the Higgs Grassmannian is a double line.
\end{example}

\begin{example} \label{ex:J2}
   The same bundle with the Higgs field $$\begin{pmatrix}0&x\\1&0
       \end{pmatrix}.$$ 
   Here $\phi$ has two Jordan blocks with different eingenvalues for $x\ne 0$; by \Cref{thm:maincurvecostante} the Higgs Grassmannian is reducible on that open subset,  with two separate sheets. At $x=0$ one has just one Jordan block, so that the fiber is a double point. Globally, the Higgs Grassmannian is a reduced irreducible double covering of the base.
\end{example}
\begin{example} \label{ex:J3}
   Again the same bundle but the Higgs field $$\begin{pmatrix}x&0\\0&-x
       \end{pmatrix}.$$ 
   Here the Jordan type is constant, with two blocks. The Higgs Grassmannian is reduced with three components: two distinct lines, each of degree one over the base, and a ``vertical'' component at $x=0$ isomorphic to $\PP^1$.
\end{example}
Note that in Examples \ref{ex:J2} and \ref{ex:J3}, differently from Example \ref{ex:J1},  there is no neighborhood of the origin where the Jordan type is constant.

\section{Higgs Grassmannians in higher dimensions}
 \label{sec:dimnrkrsingleJb}
  Recall that, given a set of commuting matrices $\{A_1,\ldots,A_s\}$, they can be upper triangularised simultaneously via a linear basis change. Moreover, this basis change  can be chosen so that one of the $A_i$'s is in Jordan form.

\begin{lemma}\label{lemma:tech}
Let $R$ be a ring, and $A, T \in \Mat_{r \times r}(R)$     commuting matrices, where $A$ is a single Jordan block with eingenvalue $\lambda$.  
Then $T$ is a  triangular Toeplitz  matrix, and there exists a polynomial
\[
p(t)=\mu_0+\mu_1t+\cdots+\mu_{r-1}t^{\,r-1}\in R[t]
\]
such that
$T = p(A)$. If $p'(\lambda)\ne 0$, then $T$ is similar to a single Jordan block. 
\end{lemma}

\begin{proof}
The first part of the proof can be found in 
\cite[Chapter VIII, \S 2.1]{GANTMACHER}. However we present here a full proof to introduce the notation needed in the second part. Set $A =\lambda I + N$,
where $N$ is a nilpotent Jordan block.
$T$ commutes with $N$, and a direct calculation shows that it can be written as
\[
T=\begin{pmatrix}
b_0 & b_1 & b_2 & \cdots & b_{r-1}\\
0   & b_0 & b_1 & \cdots & b_{r-2}\\
\vdots & \ddots& \ddots & \ddots & \vdots\\
0 & \cdots &0 & b_0 & b_1\\
0 &\cdots &0 & 0   & b_0
\end{pmatrix},
\]
i.e., it is a triangular Toeplitz matrix. As a consequence
\[
T = b_0 I + b_1N + b_2N^2+ \cdots + b_{r-1}N^{r-1}=b_0I + b_1(A-\lambda I) + \cdots + b_{r-1}(A-\lambda I)^{r-1}
\] 
which is a polynomial in $A$ of degree at most $r-1$. Expanding in $A$, we may write $T=p(A)$, and then the coefficients $b_i$ are $$b_i=\frac{p^{(i)}(\lambda)}{i!},$$ 
for $i=0,\ldots,r-1$. If $b_1=p'(\lambda)\ne 0$, since the characteristic polynomial of $T$ is $(x-b_0)^r$, having a single Jordan block   means that $\ker(B-b_0I)$ has dimension $1$. Writing $$T-b_0I=N(b_1I+b_2N+\cdots + b_{r-2}N^{r-3})$$  the matrix $U=b_1I+b_2N+\cdots + b_{r-2}N^{r-3}$ is invertible. Therefore, $\rk(NU)=\rk (N)=r-1$ which implies that $\dim\ker(T-b_0I) = 1$. 
\end{proof}

 \begin{notation}\label{nota:eqmatrici}
     Let $R$ be a ring and let $M$ be an $r\times r $ matrix with coefficients in $R$. We shall denote by $I_{M}$ the ideal
     \[
     I_{M}=\left(\sum_{k=1}^r  
z_k(M_{jk}z_i-M_{ik}z_j)=0\,,|\ 1\le i<j\le r \right)\subset R[z_1,\ldots,z_r].
     \]
     Similarly, we denote by $V_M=\Spec(R[z_1,\ldots,z_r]  / I_M)$ the  scheme defined by $I_M$. 
 \end{notation}
  
 \begin{lemma}\label{prop:containementcomuutingblock} 
    Assume that $R$ is the ring of regular functions over an open set of a variety; let $A,T\in \Mat_{r\times r}(R)$ satisfy the hypotheses of   \Cref{lemma:tech}.  Then, one has  
     \[
     I_{A}\supseteq I_{T},
     \]
    or, equivalently, $
              V_ A  \subseteq V_T$. 
Moreover, the equality holds if and only if $ p'(\lambda)\not=0$, where $p\in R[t]$ comes from \Cref{lemma:tech}. 
 \end{lemma}

 \begin{proof} According to   \Cref{lemma:onejordanblock}, the ideal $I_A$ is generated by
\begin{equation*}
    \begin{cases} 
        z_i z_{j+1} - z_j z_{i+1} & \text{if } 1\le i< j < r ,\\
         z_r z_{i+1} & \text{if } 1\le i< j = r.
    \end{cases}
    \end{equation*}
 
 Thanks to   \Cref{lemma:tech} we may assume that $T$ is a linear combination of powers of $A$. As a consequence, the generators of $I_{ T}$ obtained from \eqref{eq:mingensIHGR} rewrite as:

 \begin{equation*}
   \sum_{k=1}^{r-j} \mu_k (z_i z_{j+k} - z_j z_{i+k}) - \sum_{k=r-j+1}^{r-i} \mu_k z_j z_{i+k}, \quad 1\le i < j \le r
\end{equation*} 

 By comparing the two ideals one sees that $ I_{T} \subseteq I_{A} $.
  
 To conclude the proof we observe that if   $p'(\lambda)\not=0$ then  by \Cref{lemma:tech} $\psi$ has the same block decomposition of $A$. Therefore, there is an isomorphism $ V_A\cong V_T$. (Alternatively, one can check directly that when $p'(\lambda)\ne 0$ the two ideals   coincide.)  The equality follows then from the inclusion $V_A \subseteq V_T$ just proven. 
 \end{proof}

 \begin{example}
 For $r=3$ one has:
  \[ I_{A}=(z_1z_3-z_2^2, z_2z_3, z_3^2), \qquad I_{ T}=\bigl(\mu_1(z_1z_3-z_2^2), \mu_1z_2z_3+\mu_2z_3^2, \mu_1 z_3^2\bigr),
     \] which shows the inclusion. Moreover, when $\mu_1 \ne 0$ the two ideals agree. 
     \end{example}
     
 We consider preliminarily  the case where one of the matrices $\phi^{(1)},\ldots,\phi^{(n)}$ representing the Higgs field $\phi$ consists of single Jordan block, i.e., we prove a higher dimensional analog of \Cref{lemma:onejordanblock}.

 \begin{prop} \label{prop:n}
   Let $X$ be a smooth connected variety of dimension $n$, and $\fE=(E,\phi)$ a rank $r$ Higgs  bundle on $X$.  
   Write $\phi=(\phi^{(1)},\ldots,\phi^{(n)})$ in a local trivialization of $E$ over an open set $U \subset X$. Assume further that $\phi^{(1)}=((\lambda,r)^1)$ is a single Jordan block of rank $r$.  
   Then, the ideal of the Higgs Grassmannian over $U$ is
{ \begin{equation*} 
         I_{\Hgrass_1(\fE)} = \left( \rk\begin{pmatrix}
         z_1&\cdots & z_{r-1}\\z_2&\cdots &z_r
     \end{pmatrix}\le 1, \; z_iz_r \ |\ i=2,\ldots, r  \right)\subset \cO_{U}[z_1,\ldots,z_r].
     \end{equation*}  }  
    As a consequence each fiber of $\rho_1 \colon \hgrass{1}{\fE}\to X$ over $U$ is a curvilinear zero-dimensional scheme of degree $r$, topologically consisting of a single point.
 \end{prop}

 \begin{proof} 
     By   \Cref{prop:containementcomuutingblock} we have
     inclusions $I_{\phi^{(1)}}\supset I_{\phi^{(j)}}$ for all $j=2,\ldots,n$. As a consequence, we have
     \[
     I_{\Hgrass_1(\fE)} = \sum_{j=1}^n I_{\phi^{(j)}}=I_{\phi^{(1)}}.
     \]
     Now it suffices   to apply   \Cref{lemma:onejordanblock} to conclude the proof. 
 \end{proof}

 \begin{lemma}\label{lemma:tech2}
     Let $\phi$ be a matrix in Jordan form such that
     \begin{equation}\label{eq:phi}
 \phi =((\lambda_v,{i_v})^1)_{v=1}^s
 \end{equation}
 with $\lambda_v\not=\lambda_{v'}$ for $v\not=v'$. In particular, there is no repetition of blocks. Denote by $\phi_v= (\lambda_v,{i_v})^1$, for $v=1,\ldots,s$, the $v$-th block of $\phi$. Let $\psi$ be a matrix commuting with $\phi$. Then, $\psi$ is block diagonal (not necessarily in Jordan form) with $s$ blocks. Moreover, there are complex numbers $\mu_{j}^{(v)}\in \C$, for $1\le v\le s$ and $1\le j\le i_v$, such that the $v$-th block $\psi_v$ of $\psi$ can be written as
 \[
 \psi_v=\sum_{j=1}^{i_v}\mu_j^{(v)}\phi_v^j
 \] 
 ($\phi_v^j$ means the $j$-th power of $\phi_v$).
 \end{lemma}
 \begin{proof}
 A full proof can be found in \cite[Chapter VIII, \S 2.1]{GANTMACHER}. 
 \end{proof}

 The next Theorem describes the most general situation that we consider, namely, we assume that one of the matrices of the Higgs field has several Jordan blocks, each corresponding to a different eigenvalue. The case of repeated eigenvalues may be treated in a similar way, but  gives rise to a more complicated combinatorics.
 
 \begin{thm}\label{finalregularthm}
   Assume the hypotheses of \Cref{prop:n}, but with $\phi^{(1)}$ as in \Cref{lemma:tech2}. The  Higgs Grassmannian decomposes as a disjoint union 
     \[
     \Hgrass_1(\fE)= \coprod_{v=1}^s \Hgrass_1(E_v,\phi_{v})= \coprod_{v=1}^s V_{\phi_{v}^{(1)}},
     \]
     where $E_v\subset E$ is the invariant subbundle  with respect to the eigenvalue $\lambda_v$, $\phi_v$ is the restriction $\phi_{|E_v}$, and $\phi_{v}^{(1)}$ is the submatrix of $\phi^{(1)}$ consisting of the Jordan block corresponding to $\lambda_v$, and $V_{\phi^{(1)}}$ was defined in \Cref{nota:eqmatrici}.
    As a consequence, each fiber of $\rho_1\colon \hgrass{1}{\fE} \to X$ is a zero-dimensional scheme of degree $r$, namely, a   disjoint union of $s$ curvilinear zero-dimensional schemes of lengths $i_1,\ldots,i_s$, respectively (see \Cref{eq:phi} for the notation).  
 \end{thm}
 \begin{proof} 
    It  follows from \Cref{prop:n} and \Cref{lemma:tech2}.
 \end{proof}

\section{Structure of the Higgs Grassmannians: the case  of nonconstant Jordan type}
\label{section:discriminant} 

We focus now on the case $B=\BA^n$ and $E=\mathscr{O}_{B}^{\oplus 2}$. The Higgs field corresponds to a set of $2\times2$ matrices $\phi^{(h)}$, for $h=1,\ldots,n$, with entries in $\mathscr{O}_B$. Moreover, there is only one interesting Grassmannian to look at, namely $\hgrass{1}{\E} \subset \PP E$. We denote $(z_1, z_2)$
the vertical homogeneous coordinates in $\PP E$.
 
For instance, for $n=1$, and $\phi\not=0$, the Higgs Grassmannian is a hypersurface in $B\times \PP ^1\cong \PP E$.  We fix 
\[
0\not=\phi=\begin{pmatrix}
    0 &\phi_{12}(x)\\
    \phi_{21}(x) &\phi_{22}(x) 
\end{pmatrix},
\]
with $\phi_{ij}(x)\in H^0(B,\mathscr{O}_B)$.

\begin{lemma}\label{lemma:matrici}
Let $R$ be a domain. Let $U,V\in \Mat_{2\times2}(R)$ be two nonzero commuting matrices such that  $U_{11}=V_{11}=0$. Then, there exist two elements $u,v \in R\setminus \{0\}$ such that 
    \[
    vU=uV.
    \]
    Moreover, $\det V \not=0 $ implies $\det U\not=0$.
\end{lemma}

 \begin{proof}
We prove the first part of the statement in the case $\det V\not= 0$; the other  case is similar. We have
\[
UV-VU=\begin{pmatrix}
    U_{12}V_{21}&U_{12}V_{22}\\    U_{22}V_{21}&U_{21}V_{12}+U_{22}V_{22}
\end{pmatrix}-\begin{pmatrix}
    U_{21}V_{12}&U_{22}V_{12}\\    U_{21}V_{22}&U_{12}V_{21}+U_{22}V_{22}
\end{pmatrix},
\]
which translates into the conditions
\begin{equation}\label{eq:condts}
    U_{12}V_{21}=U_{21}V_{12},\quad U_{12}V_{22}= U_{22}V_{12},\quad U_{22}V_{21}= U_{21}V_{22}.
\end{equation}
     We have $V_{21}\not=0$ because of the assumption $\det V \not=0$. This gives an equality 
     \begin{equation}\label{eq:matU}
              U=\begin{pmatrix}
         0 & \displaystyle\frac{U_{21}V_{12}}{V_{21}} \\[10pt]  U_{21} & \displaystyle\frac{U_{21}V_{22}}{V_{21}}
     \end{pmatrix}
     \end{equation}
     in $\Mat_{2\times 2}(\mathrm{q.f.}( R))$, and we get the first part of the statement for $u=U_{21}$ and $v=V_{21}$. To conclude, just note that equations \eqref{eq:matU} and \eqref{eq:condts} imply that $\det U =0$ if and only if $U_{12}=0$ or $U_{21}=0$ if and only if $U=0$, which is a contradiction. 
 \end{proof}
Given a matrix $U$ we denote by $\gcd U$ the great common divisor of its entries.
\begin{prop}\label{prop:inddim}
    Consider a Higgs bundle $\E=(E,\phi)$, where $E=\mathscr{O}^{\oplus 2}_{\BA^n}$ is the trivial rank 2 bundle and $\phi$ corresponds  to a  nontrivial $n$-tuple $(\phi^{(1)},\ldots,\phi^{(n)})\in\End (E)^{\oplus n}\setminus\{0\}$  of commuting endomorphisms of $E$. Without loss of generality assume $\phi^{(1)}\not=0$. Suppose also that 
    \begin{equation}\label{gcd}\gcd \phi^{(h)}=1
    \end{equation} 
    for all nonvanishing $\phi^{(h)}$. 
    If $\det \phi^{(1)}\not=0$    the equation of the Higgs Grassmannian $\hgrass{1}{\E}$ is \begin{equation}\label{grassideal}
      I_{\hgrass{1}{\E}}=\left(  \phi^{(1)}_{12}z_1^2-(\phi^{(1)}_{11}-\phi^{(1)}_{22})z_1z_2-\phi^{(1)}_{21}z_2^2 \right) .
    \end{equation}
    If $\det \phi^{(1)}=0$ and $\phi^{(1)}_{12}\not=0$,   the equation of the Higgs Grassmannian $\hgrass{1}{\E}$ is 
    \[
    I_{\hgrass{1}{\E}}= \left( z_2( \phi^{(1)}_{22}z_1-\phi^{(1)}_{12}z_2)\right).
    \]
\end{prop}

\begin{proof}  
By   \Cref{lemma:indepofdiag} we may assume $\phi^{(h)}_{11}=0$, for all $h=1,\ldots,n$; moreover, by   \Cref{lemma:matrici} together with the assumption \eqref{gcd}, all the nonzero endomorphisms $\phi^{(h)}$ produce the same equation. The conclusion of the proof consists on a direct computation for the two cases.    
\end{proof}

 We describe now in detail the geometry of the Higgs Grassmannians coming from   \Cref{prop:inddim}.

 Let  $\phi=( \psi,\phi^{(2)},\ldots,\phi^{(n)})$ with $\psi\not\equiv 0$, $\psi_{11}=0$.   We distinguish two cases according to the vanishing of the determinant of   $\psi$.

If $\det\psi=0$ we have $\psi_{12}=0$ or $\psi_{21}=0$. Without loss of generality, we put $\psi_{21}=0$ and we get
\[
I_{\hgrass{1}{\E}}=z_1(\psi_{12}z_1+\psi_{22}z_2) .
\]

If $n=1$, as a consequence of the assumption $\gcd(\psi_{12},\psi_{22})=1$, 
the Higgs Grassmannian consists of two copies of $B$ meeting at $\{ b\in B\ |\ \psi_{22}(b)=0\}$;  
if $n\ge 2$, over the closed points of the nonempty subscheme of $B$ given by the ideal $(\psi_{12},\psi_{22})$, the Higgs Grassmannian also contains the fibers   
of the structure morphism $\rho_1$, that  are copies of $\PP^1$.

If $\det\psi\not=0$, which implies $\psi_{12},\psi_{21}\not=0$, denote by $\Delta_\psi$ the polynomial $$\Delta_\psi=\psi_{22}^2+4\psi_{21}\psi_{12} \in H^0(B,\mathscr{O}_B),$$ i.e. the discriminant of the equation \eqref{grassideal} defining $\hgrass{1}{\E}$. We further distinguish three subcases.
\begin{itemize}
    \item $\Delta_\psi=0$,
    \item $\Delta_\psi\in H^0(B,\mathscr{O}_B)\setminus \{0\}$ is a square,
    \item otherwise.
\end{itemize}
\paragraph{Case $\Delta_\psi=0$.}The assumption $\det\psi\not =0$ in this case also implies $\psi_{22}\not=0$.
Denote by $B_{ij}\subset B$, for $i,j\in\{1,2\}$, the open subset 
\[
B_{ij}=\{b\in B \ |\  \psi_{ij}(b)\not=0\},
\]
i.e., $B_{ij}=\Spec\C[x_1,\ldots,x_n,{\psi_{ij}}^{-1}]$. Denote also by $E_{ij}$ the restriction $E_{ij}=E{_{|B_{ij}}}$;  we have
\[
I_{\Hgrass(E_{12},\psi{_{|E_{12}}})} = (-2\psi_{12}z_1+\psi_{22}z_2)^2 
\]
and
\[
I_{\Hgrass(E_{21},\psi{_{|E_{21}}})} = (\psi_{22}z_1-2\psi_{21}z_2)^2 .
\]
The condition $\Delta_\psi=0$ ensures that the two local data agree on $B_{12}\cap B_{21}$. 

\begin{remark}
    In this case, if $n=1$, the Higgs Grassmannian is a $Z$-bundle over $B$ where $Z\cong \Spec \C[t]/t^2$. If $\dim B>1$ one has to take care of the intersection of the zero-loci of $\psi_{22}$ and $\psi_{12}\psi_{21}$, over which the fiber of the structure morphism $\rho_1$ is a $\PP^1$.  
\end{remark}

\paragraph{Case $\Delta_\psi \in H^0(B,\mathscr{O}_B)\setminus \{0\}$ is a square.}This is a natural generalization of previous point. We have
\[
I_{\Hgrass( E_{12},\psi{_{|E_{12}}})} = \left(2\psi_{12}z_1+\left(-\psi_{22}-\sqrt{\Delta_\phi}\right)z_2\right)\left(2\psi_{12}z_1+\left(-\psi_{22}+\sqrt{\Delta_\psi}\right)z_2\right)
\]
and
\[
I_{\Hgrass(E_{21},\psi{_{|E_{21}}})} =  \left(2\psi_{21}z_2+\left(-\psi_{22}-\sqrt{\Delta_\psi}\right)z_1\right)\left(2\psi_{21}z_2+\left(-\psi_{22}+\sqrt{\Delta_\psi}\right)z_1\right).
\]

 If $n=1$, the Higgs Grassmannian consists of two copies of $B$ meeting along the vanishing locus of $\Delta_\psi$; again if $n>1$, on the intersection of the zero loci of  
 $\psi_{12}$, 
 $\psi_{21}$ and $\psi_{22}$, the  
Higgs Grassmannian contains the whole fiber of $\rho_1$.

\paragraph{Case otherwise.} For simplicity, we focus on the case $n=1$; the general case is obtained as above.
From equation \eqref{grassideal} we see that the Higgs Grassmannian is irreducible and reduced. In the chart $z_2\not=0$ with coordinates $\left(x, t=\frac{z_1}{z_2}\right)$ the equation of $\hgrass{1}{\E}$ takes the form
\[
f=\psi_{12}t^2+\psi_{22} t - \psi_{21}=0;
\]
we compute the Jacobian 
\[
\nabla f =\begin{pmatrix}
    \psi_{12}'t^2+\psi_{22} 't - \psi_{21}'\\
    2\psi_{12} t +\psi_{22} 
\end{pmatrix}.
\] 
Substituting the second row $\nabla f$ into $f=0$  we have  
\[
\begin{aligned}
    \text{Sing}&\;  
    \hgrass{1}{\fE} =V(\psi_{12}'t^2+\psi_{22} 't - \psi_{21}', \psi_{22} t - 2\psi_{21},  2\psi_{12} t +\psi_{22}) \\ &=\{(x,z_1,z_2)\in\hgrass{1}{\E} \ |\psi_{12}'z_1^2+\psi_{22} 'z_1z_2 - \psi_{21}'z_2^2, \ \Delta_\psi=  \psi_{22}( 2\psi_{12} z_1 +\psi_{22}z_2)=0 \}. 
\end{aligned}
\]
We summarize the analysis that we have performed so far, again only considering the case of a one dimensional base.

\begin{prop}\label{prop:nonconstcurve}
    Let $B$ be a smooth irreducible curve and let $E$ be a locally free sheaf of rank 2 on $B$. Assume $\phi\in \Omega_B^1(\End E)$ is not in the center and $\gcd\phi_{|U} =1$ for every trivializing open $U\subset B$. Then, the morphism
  $   \rho_1\colon   \hgrass{1}{\E}  \to  B    $ 
    is finite of degree 2.
    Moreover, the Higgs Grassmannian is 
    \begin{itemize}
    \item irreducible and reduced if $\Delta_\phi\ne 0$ and is not a square in $H^0(B,\cO_B)$;
        \item reducible if $\Delta_\phi\not=0$ is a square in $H^0(B,\cO_B)$;
        \item nonreduced if $\Delta_\phi=0$. In this case, the fibers of $\rho_1$ are isomorphic to $\Spec\C[t]/t^2$.
    \end{itemize}
    In the three cases, the Higgs Grassmannian has no, two   or one section, i.e., $\fE$ has no, two, or one proper Higgs subbundle. 
    
     Moreover, if  $\gcd\phi_{|U} \not=1$ for some trivializing open $U\subset B$,   there are vertical components, so that  the morphism $\rho_1$ is not finite.\qed
\end{prop}
 
 \begin{remark}
     Although we could have formulated   \Cref{prop:nonconstcurve} in a higher-dimensional setting, we opted for this more elegant version instead. The previous analysis explains the general case step by step.
 \end{remark}

\begin{remark}
The classification of   \Cref{prop:nonconstcurve} confirms the Examples \ref{ex:J1}, \ref{ex:J2}
and \ref{ex:J3}, as  the discriminant is $0$, $4x$ and $4x^2$ in the three cases.
\end{remark}

\section{The Simpson system}
\label{section:Simpson}

Let $X$ be an $n$-dimensional smooth variety. The Simpson system of $X$  is the Higgs bundle $\gS=(S,\phi)$ on $X$
where
$$ S=\cO_X \oplus\Omega^1_X,\qquad \phi(f,\omega) = (\omega,0).$$
Note the $\cO_X$ and $\Omega^1_X$, both equipped with the zero Higgs field, are a Higgs  subbundle and a Higgs quotient of $\gS$, respectively. As noted by Simpson in \cite{simpson-unif}, when  $X$ is a projective surface and the canonical bundle $K_X$ is ample (or nef), the Simpson system is stable (semistable);
this was extended to arbitrary dimension $n$ in \cite{GKPT19}.
In both cases, the Bogomolov inequality for $\gS$ produces the Miyaoka-Yau inequality \cite{Yau77,Mi77,Mi85}
$$ 2(n+1)  c_2(X)\cdot K_X^{n-2}  \ge n  K_X^n.$$

We analyze these facts by utilizing the Higgs Grassmannians of $\gS$.

\begin{prop} \label{SimpGrass}
Let $X$ be a smooth variety of dimension $n$, and let $\gS=(S,\phi)$ be the Simpson system on $X$. Denote by $\hgrass d\gS_{\red}$   the reduction of the Higgs Grassmannian.
\begin{enumerate}
\item $\hgrass1\gS_{\red}\simeq X$, where the isomorphism is given by the section of $\PP S$ corresponding to $\cO_X$.
\item For $2 \le d \le n$ there are   isomorphisms
\begin{equation*}\label{thmsimp}
\hgrass d\gS_{\red}\cong \grass{d-1}{\Omega^1_X}\cong \grass {n+1-d} {T_X}, 
\end{equation*}
where  $\grass {d-1} {\Omega^1_X}$ and $\grass {n+1-d} {T_X}$ denote the usual  Grassmannian bundles.
\end{enumerate}
\end{prop} 
\begin{rem} \label{dualrem} To understand the isomorphism $\hgrass d\gS_{\red}\cong \grass {n+1-d} {T_X}$ let us note that given an exact sequence
\begin{equation*} \label{ranktwo } 0 \to \mathfrak F \to \gS \to \mathfrak L\to 0, 
\end{equation*}
where $\mathfrak F$ is a rank $d$ Higgs subbundle of $\gS$, with say $\mathfrak L = (L,\theta)$, then
$L^\vee$ is a rank $(n+1-d)$  subbundle of $T_X$ (note that one has an exact sequence $ 0 \to L^\vee \to T_X \to Q \to 0 $ where $Q$ sits inside $ 0 \to Q \to G^\vee \to \cO_X \to 0 $ hence it is locally free).
    \end{rem}
\begin{proof}
     Given local coordinates $(x_1,\ldots,x_n)$ on $X$ we write $\phi =\sum_{i=1}^n \phi^{(i)}\otimes dx_i$ and trivialize $S$
with local sections $(1,dx_1,\ldots,dx_n)$; then $\phi$ corresponds to the matrices
\begin{equation*}\label{matrices}
\phi^{(1)}=\begin{pmatrix} 0 & 1&0&\cdots & 0 \\ 0 & 0& 0&\cdots & 0 \\ \vdots & \vdots &\vdots&\ddots & \vdots\\ 0 & 0 &0&\cdots & 0  \end{pmatrix}, \quad 
\phi^{(2)}=\begin{pmatrix} 0 & 0&1&\cdots & 0 \\ 0 & 0& 0&\cdots & 0 \\ \vdots & \vdots &\vdots&\ddots & \vdots\\ 0 & 0 &0&\cdots & 0  \end{pmatrix}, \ \ldots,\ 
\phi^{(n)}=\begin{pmatrix} 0 & 0&0&\cdots & 1 \\ 0 & 0& 0&\cdots & 0 \\ \vdots & \vdots &\vdots&\ddots & \vdots\\ 0 & 0 &0&\cdots & 0  \end{pmatrix}.
\end{equation*}
 
We look for the equations of the reductions of the Higgs Grassmannians by characterizing the Higgs invariant subbundles of $\gS$. Let us consider the affine chart $U=\{ p_{1\ldots d}\not=0 \}\subset \grass dS \subset\PP \Lambda^d S${, where $p_{1\ldots d}$ denotes the Pl\"ucker coordinate as in \Cref{eqns}}. As usual, every point $[V]\in U$ corresponds uniquely to a matrix $M_V$ of the form
\[
M_V=(\id_d  | A_V ),
\]
where $\id_d$ is the identity matrix of size $d\times d$ and $A_V$ is a matrix of size $d\times( n+1-d )$.  
Now, a point $[V]\in U$ belongs to $\hgrass d\gS$ if and only if, for all $i=1,\ldots,n$, the columns of $\phi^{(i)} M_V^t$ span a subspace of $V$. Now, after having denoted by $M_V^{[j]}$, for $j=1,\ldots,n+1$, the $j$-th column of $M_V$, we have
\[
\phi^{(i)} M_V^t=\left( M_V^{[i+1]} \ | \  \underline{0}\right)^t,
\]
where $\underline{0}$ denotes the zero matrix of size $d\times n$. As a consequence, the point $[V]\in U$ belongs to $\grass d\gS$ if and only if, for all $i=1,\ldots,n-d+1$, the entry $(A_V)_{1i}$ vanishes. These conditions translate into the  vanishing of the Pl\"ucker coordinates $p_K$ for all totally ordered subset $K\subset \{ 2,\ldots,n+1 \}$  of cardinality $d$.

In the Grassmannian $\grass{d}{\C^{n+1}}$ the locus $\{p_K=0\}$ for $K\subset \{ 2,\ldots,n+1 \}$
is the Schubert cycle corresponding to the $d$-subplanes that contain the first basis element $e_1$; by modding out the subplanes by the subspace generated by  $e_1$ we obtain $(d-1)$-subplanes in $\C^{n} = \C^{n+1}/\C e_1$, so that 
$$\{p_K=0\} \cong \grass{d-1}{\C^{n}}.$$

As a consequence, the reduction of the Higgs Grassmannian $\hgrass{d}{\gS}_{\red}$ is isomorphic to
$\grass{d-1}{\Omega^1_X}$. For $d=1$ this is the image of the section of $\PP S$ corresponding to the subbundle $\cO_X$ of $S$.
\end{proof}
\begin{rem}
Note that the reduction in the statement is necessary. For instance, for $n=2$, given local coordinates $(x,y)$ on $X$, the Higgs field $\phi$ corresponds to the matrices
\begin{equation*}\label{matrices2}
\begin{pmatrix} 0 & 1 & 0 \\ 0 & 0 & 0 \\ 0 & 0 & 0  \end{pmatrix}, \qquad
\begin{pmatrix} 0 & 0 & 1 \\ 0 & 0 & 0 \\ 0 & 0 & 0 \end{pmatrix}.
\end{equation*}
Plugging this into equations \eqref{e:piphi} and denoting by $(z_1,z_2,z_3)$ the homogeneous fiber coordinates in $\PP S$
given by the chosen trivialization, one sees that $\grass1\gS$ is cut by the ideal
$(z_2^2, z_2z_3,z_3^2)$, whose radical is the ideal $(z_2,z_3)$ (cf.\! Example \ref{ex:r=3}).
\end{rem}

As an application we give a quick proof of the (semi)stability of the Simpson system.

\begin{thm}\label{genSimpson}
Let $X$ be an irreducible  smooth polarized variety of dimension $n$ with semistable tangent bundle, whose canonical bundle has nonnegative (positive) degree. Then the Simpson system $\gS$ of $X$ is semistable (stable).
\end{thm}
\begin{proof}
    Denoting by $H$ the polarization, the generic element  of the linear system $|mH|$ is irreducible and smooth if $m$ big enough. Iterating this procedure we can obtain a smooth irreducible curve $Y \subset X$ whose class in $H^{2(n-1)}(X,\Z)$ is $N H^{n-1}$ for a certain positive $N$. Possibly by increasing $N$, by the Mehta-Ramanathan restriction theorem \cite{Mehta-Ramanathan} the restricted tangent bundle $T_{X|Y}$ is semistable. Moreover  the bundle $T_{X|Y}$ has a nonpositive or negative degree according to whether $\deg K_X$ is nonnegative or positive.

    By functoriality, the Higgs Grassmannian of the restriction of the Simpson system of $X$ to $Y$ is 
    \begin{equation}\label{restrgrass} 
     \hgrass{1}{\gS_{|Y}} \cong Y, \qquad 
    \hgrass{d}{\gS_{|Y}} \cong \grass{n+1-d}{T_{X|Y}} \quad \text{for} \quad 2 \le d \le n.
    \end{equation}
    We prove that $\gS_{|Y}$ is a semistable (stable) Higgs bundle. Note that as $Y$ is a smooth curve it is enough to test semistability for subbundles. In accordance with  equation \eqref{restrgrass},
    $\cO_Y$ is the only rank 1 Higgs subbundle of $\gS_{|Y}$, and in view on the hypotheses on $\Omega^1_X$, it 
does not destabilize $\gS_{|Y}$. On the other hand, if    \begin{equation}\label{exactsequenceres}
    0 \to \mathfrak F \to \gS_{|Y} \to \gQ \to 0 
\end{equation}
is an exact sequence of Higgs bundles on $Y$, with  $\mathfrak F= (F,\phi_F) $ a subbundle of rank $2\le d \le n$, by the isomorphism \eqref{restrgrass}   $Q^\vee$ is a rank $(n+1-d)$ subbundle of $T_{X|Y}$, and the semistability of the latter implies 
\begin{equation}\label{degres}
    \frac{1}{n} \deg T_{X|Y} + \frac{1}{n+1-d}\deg Q \ge 0. 
\end{equation}
The conditions 
\begin{equation}\label{mu}
    \mu (F)\; {\gtreqless}\ \mu (S_{|Y})
\end{equation} 
are equivalent to $$(n+1-d) \deg T_{X|Y} + (n+1) \deg Q \; {\gtreqless}
\ 0.$$ By equation \eqref{degres} and the negativity of $T_{X|Y}$ the inequality \eqref{mu} holds true, so that $\gS_{|Y}$ is semistable (stable). Then the same is true for $\gS$.
\end{proof}

This proof is fully algebraic   if we happen to know a priori that the tangent bundle $T_X$ is semistable. On the other hand, the known proofs of the semistability of the tangent bundle of a complex smooth minimal algebraic variety are transcendental \cite{tsuji}.

\section{Relations with the spectral cover}\label{section:spectral}
We establish a relation between the spectral cover of a rank $r$ Higgs bundle
$\fE=(E,\phi)$ over a smooth irreducible projective variety $X$ and its Higgs Grassmannian
$\hgrass{1}{\fE}$. 

We recall from \cite{SimModII} the construction of the spectral cover. Denote by $T$ the total space of the cotangent bundle $\Omega^1_X$, 
with projection $p\colon T\to X$,
and by $\lambda$ the tautological section of $p^\ast \Omega^1_X$. The scheme $T$ may be considered as the relative spectrum $\Spec_X \operatorname{Sym}^\bullet T_X$. Consider
the morphism $$\Psi_{\fE}=\id \otimes \lambda-p^\ast \phi \colon p^\ast E \to p^\ast(E \otimes \Omega^1_X).$$
The spectral cover $S_\fE$ of the Higgs bundle $\fE$ is defined as the schematic support of the sheaf $\ker\Psi_{\fE}$; it is described by the equation 
\begin{equation}\label{eq:polchar}
\lambda^r + a_1 \lambda^{r-1}+\dots+a_r=0,
\end{equation}
where the $a_i$ are the canonical sections of $ \operatorname{Sym}^i \Omega_X^1$   given by the coefficients of the characteristic polymomial of $\phi$. The restriction $p \colon  S_\fE \to X$ is a  finite morphism of degree $r$ and
 $\ker\Psi_{\fE}$ restricted to $S_\fE$ is a rank 1 torsion-free sheaf  $L$; moreover, $p_\ast L$ is isomorphic to $E$. 

We define a surjective morphism $g \colon \hgrass{1}{\fE} \to S_\fE$, that set-theoretically maps an eingenspace of $\phi$ to its eigenvalue; schematically, this corresponds to the Higgs field $$\theta \in H^0(\hgrass{1}{\fE}, \rho_1^\ast\Omega^1_X)$$ 
of the universal line bundle $\gS_1$ introduced in \eqref{eq:uniSeq}, which actually takes values in $S_\fE$ as $\theta(\gS_1)\subset \rho_1^*(E\otimes \Omega_{X}^1)$.

\begin{figure}[ht]
\centering
\parbox{.45\textwidth}{ \centering
\begin{tikzpicture}
\node at (0,1) {$\ $};
    \draw[ultra thick](0,0) to (4,0);
    \draw (0,-1) to (4,-1);
    \node at (4.5,-1) {$\mathbb A^1$}; 
    \end{tikzpicture}
\caption{\label{fig:3.18} In Example \ref{ex:J1} the spectral variety has equation $(\lambda - x)^2=0$, and is isomorphic to the rank 1 Higgs Grassmannian, i.e., it is a double line.}}
\hskip8mm
\parbox{.45\textwidth}{ \centering
\begin{tikzpicture}

\draw[  thick, domain=-0.7:0.7, samples=100] 
  plot ({0.5 + 5*(\x)^2 },{\x}); 
  \node at (0.5,-1.4) {$x=0$};
 
   \draw(0,-1) to (3,-1);
    \node at (3.4,-1.2) {$\mathbb A^1$}; 
     \filldraw (0.5,-1) circle (2pt);
    \end{tikzpicture} 
\caption{In Example \ref{ex:J2} the spectral variety has equation $\lambda^2 - x=0$, is a ramified double covering of the affine line, and is isomorphic to the rank 1 Higgs Grassmannian.  \label{fig:3.19}}}
\end{figure}

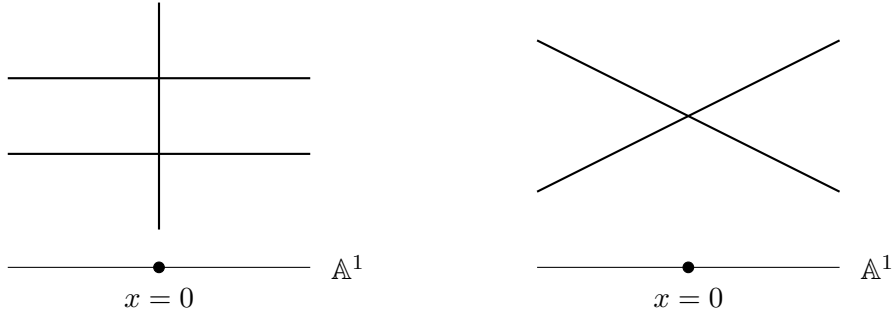
\begin{figure}[ht]
\centering
\begin{tikzpicture}
    \draw[thick](0,0-0.5) to (4,0-0.5);
    \draw[thick](0,-1-0.5) to (4,-1-0.5);
    \draw[thick](2,-2-0.5) to (2,1-0.5);
    \draw(0,-3) to (4,-3);
    \node at (4.5,-3) {$\mathbb A^1$}; 
    \node at (2,-3.4){$x=0$};
     \filldraw (2,-3) circle (2pt); 
    \draw[thick](7,0) to (11,-2);
    \draw[thick](7,-2) to (11,0);
    \draw(7,-3) to (11,-3);
    \node at (11.5,-3) {$\mathbb A^1$}; 
    \node at (9,-3.4){$x=0$};
     \filldraw (9,-3) circle (2pt);
    \end{tikzpicture}
\caption{\label{fig:3.20}   Example \ref{ex:J3}. The drawing on the left is the rank 1 Higgs Grassmannian, with a vertical component at the origin. On the right, the singular spectral curve with equation $(\lambda -x)(\lambda +x)=0$. Note that the morphism $g$ contracts the vertical component of the Higgs Grassmannian; in this case the morphism $f$ introduced in \eqref{eq:deff} is not globally defined.  }
\end{figure}
When the base $X$ is a curve, $S_\fE$ is called the spectral curve \cite{Hitchin-stable, BNR}. We consider for a while the situation where $S_{\fE}$ is smooth,
so that $L$ is a line bundle.
One has an exact sequence (see e.g.\! \cite{BNR}, Remark 3.7) 
\[
\begin{tikzcd}
     0 \arrow[r]& L \arrow[r]& p^\ast E \arrow[r,"\Psi_{\fE}"]&   p^\ast (E\otimes K_X) \arrow[r]&
L(R) \otimes p^\ast K_X \arrow[r]& 0,
\end{tikzcd}
\]
where $R$ is the ramification divisor of $S_\fE\to X$. Splitting this exact sequence in the middle we see that the quotient $p^\ast E/L$ is locally free, i.e., $L$ is a subbundle of $p^\ast E$.
By the universal property of the Higgs Grassmannian therefore we have a morphism 
\begin{equation}\label{eq:deff}
    f \colon S_\fE \to \hgrass 1\fE \quad \text{such that} \quad
\rho_1\circ f = p \quad\text{and}\quad (L,\lambda) = f^\ast \cO_{\hgrass 1\fE}(-1).
\end{equation}
The composition $g\circ f $ is the identity of $S_\fE$. In general, the morphism $f$ is only defined on the smooth locus of $S_{\fE}$.

\begin{example} \label{ex:dim2} This example shows that in dimension higher than one the relation between the Higgs Grassmannian and the spectral cover is more complicated than in the case of curves.
Let $\fE=(E,\phi)$ be a Higgs bundle on  $\BA^2$ with the Higgs field given in the coordinates $x,y$ by the matrices
\[
\phi^{(1)}=\phi^{(2)}=\begin{pmatrix}
    x&0\\0& y
\end{pmatrix}.
\]
Then, as explained in \Cref{section:discriminant}, the Higgs Grassmannian $\hgrass 1\fE$  is defined by the principal ideal
\[
I_{\hgrass 1\fE}=((x-y)z_1z_2).
\]
Therefore, it consists of three irreducible components $V,Z_1,Z_2$, where $Z_1, Z_2$ are disjoint and   map isomorphically onto $\BA^2$, while $V$ is a $\mathbb P^1$-bundle on the diagonal line  $\Delta=\{x=y\}$.

We compute the spectral cover. After letting 
$$\lambda = u \,dx + v\, dy,\qquad \phi=
\begin{pmatrix}
    x&0\\0& y
\end{pmatrix} \,(dx + dy)$$
the coefficients of the equation \eqref{eq:polchar} are
$$ a_1 = - \tr\phi = - (x+y)(dx+dy) , \qquad
a_2 = \det\phi = xy\,(dx+dy)^2,$$
so that in components the equation reads
\begin{align*}
u^2- (x+y)\,u + xy  & = 0 ,\\
 2uv- (x+y)(u+v)+2xy & = 0 ,\\
v^2 - (x+y)\, v + xy& =0\,,
\end{align*}
corresponding to the ideal
\[
I_{S_{\fE}} = ((u-x)(u- y),(v-x)(v-y),(u-v)^2).
\]
Let $Z\subset T$  and $Z_0\subset T$ be the closed subschemes corresponding to $I_{S_{\fE}}$ and $\sqrt{I_{S_{\fE}}}$, respectively (so $Z$ is the spectral cover, and $Z_0$ its reduction). The degree of  $Z\to\mathbb A^2$ is 2 away from the $\Delta$, and is 3 on it; 
the degree of  $Z_0\to\mathbb A^2$ is 2 everywhere.

The primary decomposition of $I_{S_{\fE}}$ has the three terms
\begin{align*}
I_1  = &\  (u-v,  y-v), \\
I_2  = &\  (u-v, x-v) ,\\
I_3  = &\  ((u-v)^2,  (x-v)^2, (y-v)^2, (u-v)(x+y-2v)) \,.
\end{align*}
{Note that }$I_1$ and $I_2$ are radical, while $I_3$ is not, and
$$\sqrt{I_3} = (x-y,u-v,y-v).$$
If we call $Z_i$ the component of $Z$ corresponding to $I_i$, we see that
\begin{itemize} \item $Z_1$ and $Z_2$ are reduced, and intersect over $\Delta$;
\item $Z_3$ is not reduced, and is an embedded component of $Z$, coinciding topologically with $Z_1\cap Z_2$. Note that the fibers of {$p_{|Z_3}$} are not curvilinear 0-schemes.
\end{itemize}
Again, the morphism $g$ contracts the vertical component, thus generating a singularity, but in this case it also produces nonreduced structure.
\end{example}

\begin{figure}[ht]
\centering

\begin{tikzpicture}
\begin{axis}[
    view={80}{20}, 
    width=12cm,
    height=10cm, 
    zlabel={$u=v$},
    xlabel={$x$},
    ylabel={$y$},
    grid=major,
    domain=-2:2,
    y domain=-2:2,
    xtick distance=2,
    ytick distance=2,
    ztick distance=2,
    samples=20, 
    zmin=-2, zmax=2, 
    legend pos=outer north east
]
\addplot3[
    surf, 
    shader=interp,      
    mesh/rows=25,       
    colormap={redcolormap}{color=(orange!20) color=(red!70)},
    domain=-2:2,
    y domain=0:1, 
    samples=25,
    opacity=0.6,
    faceted color=red!60!black,
    z buffer=sort
] (
    {x},                
    {x + (2-x)*y},      
    {x + (2-x)*y}       
);
\addplot3[
    surf,shader=interp,
    colormap={bluecolormap}{color=(cyan!20) color=(blue!70)},
    opacity=0.6,
    faceted color=blue!60!black
] {x}; 
\addplot3[
    surf,
    shader=interp,      
    mesh/rows=25,       
    colormap={redcolormap}{color=(orange!20) color=(red!70)},
    domain=-2:2, 
    y domain=0:1, 
    samples=25,
    opacity=0.6,
    faceted color=red!60!black,
    z buffer=sort
] (
    {x},                
    {-2 + (x+2)*y},     
    {-2 + (x+2)*y}      
);

 \addplot3[
     ultra thick,
     color=black,
     domain=-2:2,
     samples y=0 
 ] (x, x, x);

\addplot3[
     color=black,
     domain=-0.58:2,
     samples y=0 
 ] (x, x, -2);
 
 \addplot3[
     dotted,   
     color=black,
     domain=-2:-.58,
     samples y=0 
 ] (x, x, -2);
 
\end{axis}
\node at (12,4.1,8) {$\Delta$};
\end{tikzpicture}
\caption{The spectral cover of \Cref{ex:dim2}. The pink and blue planes are the $Z_1$ and $Z_2$ components, respectively; they lie in the hyperplane $u=v$. The black line is the support of the embedded component $Z_3$.}
\end{figure}
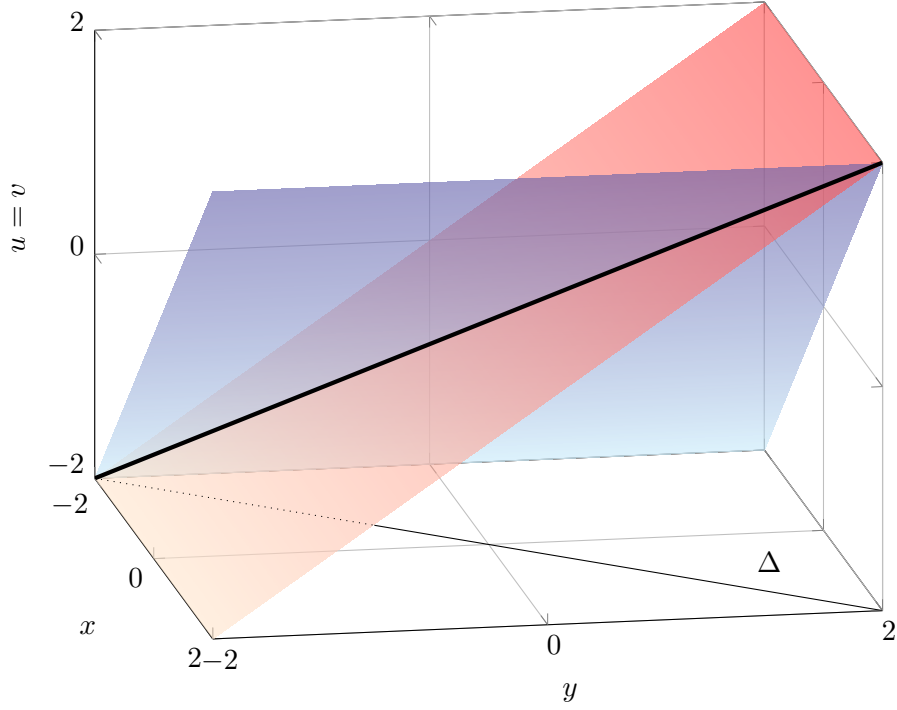

\begin{example}
    Let us consider the Higgs bundle $\fE = (E,\phi)$ on $\mathbb A^n$ given by the diagonal Higgs field
    \[
    \phi =\begin{pmatrix}
        x_1dx_1 &0 & \cdots & 0\\
        0&x_2dx_2 & \cdots & 0\\
        \vdots&\vdots&\ddots&\vdots\\
        0&0&\cdots &x_ndx_n
    \end{pmatrix}.
    \]
A direct computation via \eqref{e:piphi} shows that
\[
I_{\hgrass1\fE}= \left(z_iz_jx_i\ |\ 1\le i\not= j\le n\right).
\]
This ideal is radical as it is monomial, generated by squarefree monomials. Its decomposition in minimal primes is
\begin{equation}\label{eq:miprimeex}
    I_{\hgrass1\fE}= \bigcap_{i=1}^n(z_j\ |\ j\not=i) \ \cap\bigcap_{\substack{P\subset\{1,\ldots,n\}\\ |P|> 1 }}(x_i\ |\ i\in P)+(z_i\ |\ i\in P^c),
\end{equation}
where $P^c\subset \{1,\ldots,n\}$ denotes the complement of $P$. To see this, one   first observes that $I_{\hgrass1\fE}\subset \bigcap_{i=1}^n(z_j\ |\ j\not=i)$, then saturates by this intersection and proceeds recursively, see \cite[\S 4.4 Theorem 10]{COXIdeals}.

As a consequence, the Higgs Grassmannian has $2^n-1$ irreducible components that we divide in two collections.
\begin{itemize}
    \item The components corresponding to the first $n$ ideals in the intersection \eqref{eq:miprimeex} are ``horizontal,'' that is, they are disjoint and map isomorphically onto $\BA^n$ via the respective restrictions of the structure morphism.
    \item The other components are ``vertical,''  as they are supported each on a linear coordinate subspace of codimension greater than or equal to 2. More precisely, the generic fiber {of $\rho_1$} over a subspace of codimension $c$ is isomorphic to $\mathbb P^{c-1}$.
\end{itemize}
 
The number $2^n-1$ is then obtained by summing up all contributions. 

We move now to the description of the spectral cover. We put first $\lambda=\sum_{i=1}^n\lambda_id x_i$ and we study the scheme defined by the vanishing of \eqref{eq:polchar}. Consider, for $i=1,\ldots,n$, the prime ideals 
\begin{equation}\label{eq:gammas}
I_i=(\lambda_i-x_i)+( \lambda_j\ | \ j\not=i ).
\end{equation}

By the very definition of characteristic polynomial,   we have
\begin{equation}\label{eq:radI}
\sqrt{I_{S_{\fE}}}= \bigcap_{i=1}^nI_i.
\end{equation}
However, the spectral cover is not reduced as one can directly check. Note that the morphism   from the Higgs Grassmannian to the spectral cover contracts {the} $2^n-n-1$ {vertical} irreducible components.

We spell out in detail the case $n=2$. Setting for simplicity $$x=x_1,\quad y=x_2,\quad u = \lambda_1, \quad v=\lambda_2,$$ we have
$$ I_{S_{\fE}} = (u^2-xu, 2uv-xv-yu+xy,v^2-yv).$$
The primary decomposition
$ I_{S_{\fE}} = I_1\cap I_2\cap I_3$
shows two reduced components $Z_1$, $Z_2$ corresponding to
$$I_1=(v,x-u),\qquad I_2 = (u,y-v),$$
and a nonreduced component $Z_3$ corresponding to
$$I_3 = (x^2,y^2,u^2,v^2,xu,yv,xy-yu-xv+2uv).$$
Since $\sqrt{I_3}=(x,y,u,v)$, the reduction of $Z_3$ is just a point (the intersection of $Z_1$ and $Z_2$).

On the other hand, one has
$$\sqrt{I_{S_{\fE}}} = I_1\cap I_2,$$ 
in conformity with \Cref{eq:gammas,eq:radI}.
\end{example}

\section{Higgs flag scheme and Higgs Quot scheme }
\label{sec:flag}
This section serves as a suggestion for possible further directions in the study of moduli spaces analogous to the Higgs Grassmannians. Precisely, we define now the Higgs versions of the flag variety and of the Quot scheme, and     present some basic example.

\begin{defin}
    Let $\fE=(E,\phi)$ be a Higgs bundle of rank $r\ge 0$ over a variety $X$, and let $\boldsymbol{r}=(0\le r_1\le \cdots \le r_k\le r)\in\mathbb Z_{\ge 0}^k$ be a nondecreasing sequence of nonnegative integers. The Higgs flag scheme $\fFl_{\boldsymbol{r}}(\fE)$ is the scheme representing the functor
    \[
    \begin{tikzcd}[row sep =tiny]
        \text{Sch}^{\mathrm{op}}  \arrow[r,"F_{\boldsymbol{r},\fE}"]&   \text{Sets}\\ 
        T\arrow[r,mapsto] & \left\{ \fE_{r_1}\subset  \dots \subset  \fE_{r_k} \subset  p_2^*\E   \right\},
    \end{tikzcd}
    \] 
where the $\fE_{r_i}$'s are rank $r_i$ Higgs subbundles of $p_2^\ast\fE$, and 
$p_2\colon T\times X \to X$ is the projection onto the second factor. 
\end{defin}
 
\begin{remark}\label{rem:setflag}  Note that the Higgs flag scheme comes equipped with a morphism $\varphi_{\boldsymbol{r}}:\fFl_{\boldsymbol{r}}(\fE)\to X$ in analogy with the morphism $\rho_d:\hgrass d\fE\to X$ introduced in \eqref{eq:uniSeq}. Set theoretically, the scheme $\fFl_{\boldsymbol{r}}(\fE)$  parametrizes flags $0\subset E_1\subset \cdots\subset E_k\subset E $ of $\phi$-invariant subbundles of respective rank $r_i$, for $i=1,\ldots,k$, of $E$.
\end{remark}
\begin{thm}\label{thm:flagexist}
    For every rank $r$ Higgs bundle  $\fE=(E,\phi)$ over a variety $X$ and every nondecreasing sequence of nonnegative integers $\boldsymbol{r}=(0\le r_1\le \cdots \le r_k\le r)$, the Higgs flag scheme $\fFl_{\boldsymbol{r}}(\fE)$ exists and  is unique up to canonical isomorphism.
\end{thm}
\begin{proof}
    The proof is analogous to that of the existence and universality of the classical flag variety. It consists in constructing the Higgs flag scheme as a closed subscheme of the product $\prod_{i=1}^k\hgrass{r_i} \fE$, see \cite[Part III.9.1]{FULTON}. Alternatively, one can realize the Higgs flag scheme as a closed subscheme of the classical flag variety by arguing as in the computation of the equations of the Higgs Grassmannian inside the classical Grassmannian,  see \Cref{rem:setflag} and \Cref{eqns}.
\end{proof}

\begin{example} \label{ex:flag}
    Let us give explicit equations for the flag scheme $\fFl_{1,2}(\fE)$ on $\mathbb A^1$ where $\fE =  (E,\phi)$ for some Higgs field $\phi$ on the trivial rank 3  bundle $E=\mathscr{O}_{\BA^1}^{\oplus 3}$, see also Examples \ref{eq23} and \ref{ex:r=3}. We introduce vertical homogeneous coordinates $[z_1:z_2:z_3]$ and $[y_1:y_2:y_3]$ on $\mathbb P E$ and $\mathbb P E^\vee$, so that we have
    \[
    I_{\hgrass1\fE} =\left( \left. 
\sum_{k=1}^3   z_k(\phi_{jk}z_i-\phi_{ik}z_j)=0\,\right|\, 1\le i<j\le 3 \right), \]
and
\[ I_{\hgrass2\fE}=\left( \left. 
\sum_{k=1}^3   y_k(\phi_{kj}y_i-\phi_{ki}y_j)=0\,\right|\, 1\le i<j\le 3 \right).
    \]
    Note that   a point $[\Span (v) ]\in \hgrass1\fE$ corresponds to the $\phi$-invariant line $\Span (v)=\phi (\Span (v)) $ and a point $[ \ell ]\in  \hgrass2\fE$, for $\ell \in \PP E^\vee$ corresponds to the plane $\ker ( \ell \circ\phi)=\ker (\ell) $. 
    In other words, if we put
    \[
    f_{\fFl} = y_1z_1+y_2z_2+y_3z_3 \in \mathbb C[y_1,y_2,y_3,z_1 , z_2, z_3],
    \]
    then the closed immersion $\fFl_{1,2}(\fE)\subset \mathbb P E\times \mathbb P E^\vee$ corresponds to the {bihomogeneous} ideal 
    \[
    I_{\fFl_{1,2}(\fE)}= I_{\hgrass1\fE}+ I_{\hgrass2\fE}+(f_{\fFl_1}),
    \]
    where we interpret the homogeneous ideals $I_{\hgrass1\fE}$ and $I_{\hgrass1\fE}$ as ideals in $\mathbb C[y_1,y_2,y_3,z_1 , z_2, z_3]$. 
We describe explicitly the geometry of the fibers of the morphism $\varphi_{1,2}:\fFl_{1,2}(\fE)\to \BA^1 $ in some cases.  Suppose that $\phi$ corresponds to a matrix having one of the following forms:
\[
A=\begin{pmatrix}
    0&1&0\\
    0&0&1\\
    0&0&0
\end{pmatrix},\quad
B=\begin{pmatrix}
    0&1&0\\
    0&0&0\\
    0&0&\alpha
\end{pmatrix},\quad
C=\begin{pmatrix}
    0&0&0\\
    0&\beta&0\\
    0&0&\gamma
\end{pmatrix},
\]
where, for simplicity, we assume $\alpha,\beta,\gamma,\beta-\gamma\in \mathbb C^*$ to be nonzero constants. In \Cref{tab:placeholder} we distinguish the analysis in three cases according to the type of matrix associated with $\phi$; the description comes from a direct computation. Note that in all cases the fibers of $\varphi_{1,2}$ have length 6. The fiber in case (A) is the fourth of the list  in \cite{BRIANCON} and the fifth of the list  in \cite{POONEN}.
\end{example}
\begin{table}[ht]
    \centering
    \begin{tabular}{c|c|c|c}
    & $ I_{\hgrass1\fE} $&$ I_{\hgrass2\fE} $&Fiber of $\varphi_{1,2}$ \\[0.6em]
    \hline \vphantom{\Bigg|}
        (A) & $(z_2^2-z_1z_3, z_2z_3, z_3^2)$ & $(y_1^2, y_1y_2, y_2^2-y_1y_3)$& $\Spec(\mathbb C [t_1,t_2]/(t_1(t_1-t_2),t_2^3)$  \\[1em]
        (B) & $(z_2^2, (z_1\alpha-z_2)z_3, z_2z_3 )$  & $(y_1^2, y_1y_3, (y_2\alpha-y_1)y_3)$ & three copies of $\Spec(\mathbb C[t]/t^2)$ \\[1em]
        (C) & $(z_2z_1 , z_1z_3 , z_2z_3 )$ & $(y_1y_2, y_1y_3, y_2y_3)$ & six distinct closed points \\[10pt]
    \end{tabular}
    \caption{The ideals of the Higgs Grassmannians and the description of the Higgs flag in \Cref{ex:flag}}
    \label{tab:placeholder}
\end{table}

\begin{example}  
Let $X$ be a smooth variety of dimension $n$, and let $\gS=(S,\phi)$ be the Simpson system of $X$, see \Cref{section:Simpson}. Let also $\boldsymbol{r}=(0<r_1\le\cdots \le r_k\le n+1)\in\mathbb Z_{>0}^k$ be a nondecreasing sequence of positive integers. Denote by $\fFl_{\boldsymbol{r}}(\gS)_{\red}$   the reduction of the Higgs flag scheme of $\gS$. Then, arguing as in the proof of  \Cref{SimpGrass}, one has 
\[
\fFl_{\boldsymbol{r}}(\gS)_{\red}\cong \Fl_{\boldsymbol{r}'}{(\Omega^1_X)},
\]
   where $\boldsymbol{r}'=(0\le r_1-1\le\cdots \le r_k-1\le n) $, and $\Fl_{\boldsymbol{r}'}{(\Omega^1_X)} $ denotes the classical flag variety of the cotangent bundle.
   
\end{example}

 When studying the stability of a Higgs bundle, one has to consider all subsheaves, not only subbundles; that is, one needs to take into account the cases where the corresponding quotient is
 not locally free. For this reason, Higgs Grassmannians do not encode 
enough information about stability. A particularly useful tool to this end is the following Higgs equivariant version of the Quot scheme.

\begin{defin}
    Let $X$ be a smooth projective variety, and let $\fE=(E,\phi)$ be a Higgs bundle on $E$. The Higgs  Quot scheme $\Quot_{\fE}$ is the scheme representing the functor 
    \[
    \begin{tikzcd}[row sep =tiny]
        \text{Sch}^{\mathrm{op}}  \arrow[r,"\underline{\Quot}_{\,\fE}"]&   \text{Sets}\\
        T\arrow[r,mapsto] & \{ \mathfrak F\in\underline{\Quot}_{\,E}(T)\ |\ \mathfrak F\subset p_2^*\fE \mbox{ is a Higgs subsheaf }  \},
    \end{tikzcd}
    \] 
    where $\underline{\Quot}_{\,E}$ is the classical Quot functor\footnote{The definition of Quot functor is standard and we refer to \cite{FGAexplained} for more details.} of $E$, and 
$p_2\colon T\times X \to X$ is the projection onto the second factor.
\end{defin} 
\begin{thm}
    If $X$ is a smooth projective variety, and   $\fE=(E,\phi)$ is a Higgs bundle on $E$, the Higgs  Quot scheme $\Quot_{\fE}$ exists and  is unique up to canonical isomorphism.
\end{thm}
\begin{proof}
    The proof is standard and similar to that of \Cref{thm:flagexist}. It just consists in noticing that $\underline{\Quot}_{\,\fE}$ is a closed subfunctor of $\underline{\Quot}_{\,E}$, and hence   is represented by a closed subscheme of ${\Quot}_{E}$. 
\end{proof}
\begin{remark}
    Set theoretically, the Higgs Quot scheme parametrizes $\phi$-invariant subsheaves of the locally free sheaf $E$. It is a scheme locally of finite type and  decomposes as a disjoint union of   projective (nonnecessarily connected) schemes
    \[
    \Quot_{\fE}=\coprod_{P\in\mathbb Z_{\ge0}[t]}\Quot_{\fE}^P,
    \]
    where the closed points of $\Quot_{\fE}^P$ correspond to Higgs subsheaves $F\subset E$ such that the quotient $E/F$ has Hilbert polynomial $P$. We  will refer to this open subschemes of $\Quot_{\fE}$ as Quot schemes as well.
\end{remark}
\begin{remark}
    The hypothesis that  $X$ is projective is necessary  to define the Higgs Quot scheme $\Quot_{\fE}^P$ only when $\deg P>0$. Therefore, when the polynomial $P$ is a nonnegative constant and $\fE$ is a Higgs bundle on a quasi projective variety, we will consider the Quot scheme $\Quot_{\fE}^P$. In this setting the Quot scheme is clearly only quasi-projective and not projective. 
\end{remark}
\begin{example}
    In this last example we give an explicit description of the schemes $\Quot^d_{\gS}$, where $\gS=(S,\phi) $ is the Simpson system on $\BA^1$ and $d$ is a positive integer. Recall that  $S\cong \cO_{\BA^1}^{\oplus 2}$ and in local coordinates $\phi $ corresponds to the matrix
    \[
    \phi=\begin{pmatrix}
        0&1\\0&0
    \end{pmatrix}.
    \]
    Since we work over the affine scheme $\BA^1$, we switch to the more convenient language of commutative algebra. Denote by $R=\mathbb C[x]$ the ring of regular functions on $\BA^1$.  
    A closed point $[K]\in \Quot_{S}^d  $ corresponds to a  free submodule $K\subset R^{\oplus 2}$ of rank 2 with  $\dim_{\mathbb C}(R^{\oplus 2}/K )=d$. In particular, it can be written as
    \[
    K=( p_1(x)e_1+q_1(x)e_2, p_2(x)e_1+q_2(x)e_2 ),
    \]
    for some $p_1,p_2,q_1,q_2\in R$.
    
    Now, the point $[K]\in \Quot_{S}^d$ belongs to $\Quot_{\gS}^d$ if and only if $\phi(K) \subset K\otimes \Omega_{\BA^1}^1$. In commutative algebra language, this forces  
    \[
    q_1 =0,\ q_2\in (p_1)\quad\mbox{ or }\quad q_2 =0,\ q_1\in (p_2).
    \]
    In other words, for every submodule $K\subset R^{\oplus 2}$ corresponding to a closed point $[K]\in \Quot_{\gS}^d$ we can choose canonical generators $K=(f,g)$ with
    \[
    f= p_1(x)e_1\quad \mbox{ and }\quad g= p_2(x)e_1+q(x)e_2 ,
    \]
    where $q\in(p_1)$ and $\deg p_2<\deg p_1$. Consider the canonical projection onto the factors
    \[
    \begin{tikzcd}[row sep=tiny]
        R^{\oplus 2}\arrow[r,"\pi_i"]&R\\
        u_1e_1+u_2e_2\arrow[mapsto,r]& u_i e_i.
    \end{tikzcd}
    \]
As already observed, the pair $(I_1,I_2)= \left((\pi_1 f),( \pi_2 g)\right) $ has the property $I_1\supset I_2$ and therefore  defines a point $[I_1,I_2]\in(\BA^1)^{[a,b]}$ for some $a,b\in \mathbb Z_{\ge 0}$ with $a+b=d$; here $(\BA^1)^{[a,b]}$ denotes the nested Hilbert scheme of points, i.e., the moduli space of nestings of zero-dimensional subschemes of $\BA^1$ of respective lengths $a\le b$. See \cite{SERNESI} for the definition and \cite{DNHS} for the notation.
    
    Note that, since we only used the property that $R$ is a PID, this construction works verbatim in families. Hence, by universality we have a morphism
    \[
    \begin{tikzcd}[row sep = tiny]
        \Quot_{\gS}^d\arrow[r,"\Phi_{\gS}"]&\displaystyle\coprod_{\substack{a+b=d\\ 0\le a\le b\le d}}(\BA^1)^{[a,b]}\cong \displaystyle\coprod_{b=\lceil d/2\rceil}^d\BA^b,
    \end{tikzcd}
    \]
    whose fiber over a point $[I,J]\in (\BA^1)^{[a,b]}$ is an affine space of dimension $a-1$. In particular,   the two spaces have the same number of connected components, which is $\lfloor d/2\rfloor+1$. Note also that the association
    \[ 
    \begin{tikzcd}[row sep = tiny]
        \displaystyle\coprod_{a+b=d}(\BA^1)^{[a,b]}\arrow[r]&\Quot_{\gS}^d \\
         {[(u(x)),(v(x))]}\arrow[r,mapsto]& {[ (u(x)e_1 , v(x) e_2) ]},
    \end{tikzcd}
    \]
can be performed in families and   defines a section of $\Phi_{\gS}$. This trivializes the morphism $\Phi_{\gS}$ and provides the isomorphism
\[
\Quot_{\gS}^d\cong \displaystyle\coprod_{b=\lceil d/2\rceil}^d  \BA ^{d-1}
\]
and the   equidimensionality of $\Quot_{\gS}^d$. 
\end{example}

\let\oldthebibliography\thebibliography
\let\endoldthebibliography\endthebibliography
\renewenvironment{thebibliography}[1]{
  \begin{oldthebibliography}{#1}
    \setlength{\itemsep}{2pt}    
    \setlength{\parskip}{0pt}    
}{
  \end{oldthebibliography}
}


\begin{thebibliography}{10}

\bibitem{BNR}
{\sc A.~Beauville, M.~S. Narasimhan, and S.~Ramanan}, {\em Spectral curves and
  the generalised theta divisor}, J. Reine Angew. Math., 398 (1989),
  pp.~169--179.

\bibitem{BRIANCON}
{\sc J.~Brian\c{c}on}, {\em Description de $\text{Hilb}^n\text{C}\,\{x,y\}$},
  Invent. Math., 41 (1977), pp.~45--89.

\bibitem{bruzzo-grana-adv}
{\sc U.~Bruzzo and B.~Gra\~{n}a Otero}, {\em Semistable and numerically
  effective principal ({H}iggs) bundles}, Adv. Math., 226 (2011),
  pp.~3655--3676.

\bibitem{bruzzo-hernandez-dga}
{\sc U.~Bruzzo and D.~Hern{\'a}ndez~Ruip{\'e}rez}, {\em Semistability vs.
  nefness for ({H}iggs) vector bundles}, Differential Geom. Appl., 24 (2006),
  pp.~403--416.

\bibitem{COXIdeals}
{\sc D.~A. Cox, J.~Little, and D.~O'Shea}, {\em Ideals, varieties, and
  algorithms---an introduction to computational algebraic geometry and
  commutative algebra}, Undergraduate Texts in Mathematics, Springer, Cham,
  fifth~ed., 2025.

\bibitem{FGAexplained}
{\sc B.~Fantechi, L.~G\"ottsche, L.~Illusie, S.~L. Kleiman, N.~Nitsure, and
  A.~Vistoli}, {\em Fundamental algebraic geometry}, vol.~123 of Mathematical
  Surveys and Monographs, American Mathematical Society, Providence, RI, 2005.
\newblock Grothendieck's FGA explained.

\bibitem{FULTON}
{\sc W.~Fulton}, {\em Young tableaux}, vol.~35 of London Mathematical Society
  Student Texts, Cambridge University Press, Cambridge, 1997.
\newblock With applications to representation theory and geometry.

\bibitem{GANTMACHER}
{\sc F.~R. Gantmacher}, {\em The theory of matrices. {V}ol. 1}, AMS Chelsea
  Publishing, Providence, RI, 1998.
\newblock Translated from the Russian by K. A. Hirsch, Reprint of the 1959
  translation.

\bibitem{DNHS}
{\sc M.~Graffeo, P.~Lella, S.~Monavari, A.~T. Ricolfi, and A.~Sammartano}, {\em
  The geometry of double nested {H}ilbert schemes of points on curves}, Trans.
  Amer. Math. Soc., 378 (2025), pp.~6013--6047.

\bibitem{GKPT19}
{\sc D.~Greb, S.~Kebekus, T.~Peternell, and B.~Taji}, {\em The {M}iyaoka-{Y}au
  inequality and uniformisation of canonical models}, Ann. Sci. \'Ec. Norm.
  Sup\'er. (4), 52 (2019), pp.~1487--1535.

\bibitem{Hitchin-self}
{\sc N.~J. Hitchin}, {\em The self-duality equations on a {R}iemann surface},
  Proc. London Math. Soc. (3), 55 (1987), pp.~59--126.

\bibitem{Hitchin-stable}
\leavevmode\vrule height 2pt depth -1.6pt width 23pt, {\em Stable bundles and
  integrable systems}, Duke Math. J., 54 (1987), pp.~91--114.

\bibitem{Mehta-Ramanathan}
{\sc V.~B. Mehta and A.~Ramanathan}, {\em Restriction of stable sheaves and
  representations of the fundamental group}, Invent. Math., 77 (1984),
  pp.~163--172.

\bibitem{Mi77}
{\sc Y.~Miyaoka}, {\em On the {C}hern numbers of surfaces of general type},
  Invent. Math., 42 (1977), pp.~225--237.

\bibitem{Mi85}
\leavevmode\vrule height 2pt depth -1.6pt width 23pt, {\em The {C}hern classes
  and {K}odaira dimension of a minimal variety}, in Algebraic geometry, Sendai,
  1985, vol.~10 of Adv. Stud. Pure Math., North-Holland, Amsterdam, 1987,
  pp.~449--476.

\bibitem{Mukai-inv}
{\sc S.~Mukai}, {\em An introduction to invariants and moduli}, vol.~81 of
  Cambridge Studies in Advanced Mathematics, Cambridge University Press,
  Cambridge, 2003.
\newblock Reprinted 2006.

\bibitem{POONEN}
{\sc B.~Poonen}, {\em Isomorphism types of commutative algebras of finite rank
  over an algebraically closed field}, in Computational arithmetic geometry,
  vol.~463 of Contemp. Math., Amer. Math. Soc., Providence, RI, 2008,
  pp.~111--120.

\bibitem{SERNESI}
{\sc E.~Sernesi}, {\em Deformations of algebraic schemes}, vol.~334 of
  Grundlehren der mathematischen Wissenschaften [Fundamental Principles of
  Mathematical Sciences], Springer-Verlag, Berlin, 2006.

\bibitem{simpson-unif}
{\sc C.~T. Simpson}, {\em Constructing variations of {H}odge structure using
  {Y}ang-{M}ills theory and applications to uniformization}, J. Amer. Math.
  Soc., 1 (1988), pp.~867--918.

\bibitem{simpson-local}
\leavevmode\vrule height 2pt depth -1.6pt width 23pt, {\em Higgs bundles and
  local systems}, Inst. Hautes \'Etudes Sci. Publ. Math., 75 (1992), pp.~5--95.

\bibitem{SimModII}
\leavevmode\vrule height 2pt depth -1.6pt width 23pt, {\em Moduli of
  representations of the fundamental group of a smooth projective variety.
  {II}}, Inst. Hautes \'Etudes Sci. Publ. Math.,  (1994), pp.~5--79.

\bibitem{tsuji}
{\sc H.~Tsuji}, {\em Stability of tangent bundles of minimal algebraic
  varieties}, Topology, 27 (1988), pp.~429--442.

\bibitem{Yau77}
{\sc S.-T. Yau}, {\em Calabi’s conjecture and some new results in algebraic
  geometry}, Proc. Nat. Acad. Sci. U.S.A, 74 (1977), pp.~1798--1799.

\end{thebibliography}
\end{document}